\documentclass[a4paper,12pt, reqno]{amsart}
\usepackage{stmaryrd} %\usepackage{pdfsync}
\usepackage[numbers]{natbib}
\usepackage{float}
\usepackage{graphicx}
\usepackage{url}
\usepackage{amssymb,amsmath,amsthm,verbatim,longtable,latexsym}
\usepackage{mathabx}
\usepackage[all,ps]{xy}
\usepackage{color}
\usepackage{hyperref}
\usepackage{times}
\usepackage{todonotes}
\usepackage{bbold}
\numberwithin{equation}{section}

\usepackage{psfrag}

\newcommand{\rsdraw}[3]{\raisebox{-#1\height}{\scalebox{#2}{\includegraphics{#3.eps}}}}

\newcommand{\co}{\colon}
\newcommand{\tens}{\otimes} 
\newcommand{\id}{\mathrm{id}}

\newcommand{\Mod}{\mathrm{Mod}}
\newcommand{\mc}{\mathcal}
\newcommand{\uu}{\Bbb1}
\newcommand{\Ob}{\mathrm{Ob}}

\newcommand{\cA}{\mathcal{A}}
\newcommand{\cB}{\mathcal{B}}
\newcommand{\cC}{\mathcal{C}}

\newcommand{\cX}{\mathcal{Y}}

\newcommand{\rA}{\mathrm A}
\newcommand{\rB}{\mathrm B}
\newcommand{\rC}{\mathrm C}
\newcommand{\rD}{\mathrm D}

\newcommand{\rF}{\mathrm F}
\newcommand{\rG}{\mathrm G}
\newcommand{\rH}{\mathrm H}
\newcommand{\rK}{\mathrm K} % comparison functor
\newcommand{\rL}{\mathrm L} % list monad
\newcommand{\rM}{\mathrm M}
\newcommand{\rN}{\mathrm N}
\newcommand{\rQ}{\mathrm Q}

\newcommand{\rS}{\mathrm S}
\newcommand{\rT}{\mathrm T}
\newcommand{\rU}{\mathrm U}
\newcommand{\rV}{\mathrm V}

\newcommand{\bA}{\mathbb A}
\newcommand{\bB}{\mathbb B}
\newcommand{\bC}{\mathbb C}
\newcommand{\bD}{\mathbb D}

\newcommand{\bQ}{\mathbb Q}
\newcommand{\bP}{\mathbb P}
\newcommand{\bL}{\mathbb L}
\newcommand{\bS}{\mathbb S}
\newcommand{\bT}{\mathbb T}

\newcommand{\cCat}{\mathsf{Cat}} % categories
\newcommand{\cSet}{\mathsf{Set}} % sets
\newcommand{\cSgp}{\mathsf{SemiGp}} % semigroup

\newtheorem{thm}{Theorem}[section]
\newtheorem{prop}[thm]{Proposition}
\newtheorem{lem}[thm]{Lemma}
\newtheorem{cor}[thm]{Corollary}

\theoremstyle{definition}
\newtheorem{defn}[thm]{Definition}
\newtheorem{exa}[thm]{Example}

\theoremstyle{remark}
\newtheorem{rem}[thm]{Remark}

\title[
Duplicial functors, descent categories and Hopf modules]{Duplicial functors, descent categories and generalized Hopf modules
}

\author{Ivan Bartulovi\'c}
\address{
    TU Dresden,
    Institute of Geometry,
    01062 Dresden,
    Germany
}
\email{ivan.bartulovic@tu-dresden.de}
\author{John Boiquaye}
\address{Department of Mathematics, 
University of Ghana, 
Botanical \newline \indent Gardens Road,
 Legon, GA-489-9348, 
 P.O. Box LG 62, Accra, 
 Ghana}
\email{jboiquaye@ug.edu.gh}
\author{Ulrich Krähmer}
\address{
    TU Dresden,
    Institute of Geometry,
    01062 Dresden,
    Germany
}
\email{ulrich.kraehmer@tu-dresden.de}
\begin{document}

\begin{abstract} 
Böhm and Ştefan have expressed 
cyclic homology as an invariant 
that assigns homology
groups $\mathrm{HC}^\chi_i(\rN,\rM)$
to right and left coalge\-bras 
$\rN$ respectively $\rM$ over 
a  distributive
law $\chi $ between two 
como\-nads. 
For the key  
example associated to a bialgebra 
$H$, right
$\chi$-coalgebras have a 
description in terms of
modules and comodules over 
$H$. 
The present
article formulates conditions
under which such a description
is simultaneously possible for
the left 
$\chi$-coalgebras. In the
above example, this is the
case when the bialgebra $H$ is
a Hopf algebra with bijective
antipode. 
We also discuss how
the generalized Hopf module
theorem by Mesablishvili and
Wisbauer features both in
theory and examples. 
\end{abstract}

\maketitle

\setcounter{tocdepth}{1} %put this in comment if you want to see the skeleton of the paper
\tableofcontents 

\section{Introduction}

\subsection{Simplicial and
cyclic objects from comonads}  
The classical homology theories
studied in algebra and geometry such
as group, Lie algebra, Hochschild,
or singular homology can
be expressed as the homology of
a suitable comonad $\bT$ on some
category $\cB$,
see~\cite[Section~8.7]{weibel}: 
given any functors
$\rM \co \mc Y \to
\mc B$ and $\rN \co
\mc B \to \mc Z$, one builds a
simplicial functor  
$\rC_{\bT}(\rN, \rM)_n := 
\rN\rT^{n+1}\rM$, where~$\rT$ is the
endofunctor underlying $\bT$. 
Typi\-cally,
$\mc Y$ is the terminal category and 
$\mc Z$ is the category of 
abelian groups, so    
$\rC_{\bT}(\rN, \rM)$ is just a
simplicial abelian group. 
The homology theory is defined in
terms of the associated chain
complex \cite[Definition 8.3.6]{weibel}. 

Cyclic homology theories~\cite{Con83} 
follow a more general pattern:
as shown by Böhm and \c
Stefan~\cite{BS08},
they define invariants of 
 a comonad distributive law~$\chi \colon 
\bT \bS \Rightarrow \bS
\bT$, or,
equivalently, of a comonad
with a given factorization as
a composition~$\bT\bS$ of two
comonads (see
Section~\ref{distlawsandcoalgs}). 
The coef\-ficients of the theory
are left respectively
right~$\chi $-coalgebras~$\lambda \co \rN \rS \Rightarrow \rN \rT$
    and~$\rho \co \rT \rM
\Rightarrow \rS \rM$
(see
Definition~\ref{chicoalgdef}
for this notion). In the
presence of these extra
structures, the simplicial functor 
$\rC_{\bT}(\rN, \rM)$ becomes a 
duplicial functor 
(see \cite{DK85} for this 
generalization of
Connes' cyclic objects due to Dwyer
and Kan), and 
the associated chain complex
becomes simultaneously a cochain complex.
Thus it is the presence of the
second comonad $\bS$ that leads to the coboundary map which is used to define cyclic homology.

\subsection{Background and
aim}
Assume 
$\rF\co \mc A \rightleftarrows
\mc B \co \rU$ is an adjunction for 
$\bT$, meaning that 
$\rT = \rF\rU$
(see~\cite[Main
Application 8.6.2.]{weibel}). 
If $\bS$ is a lift of a comonad 
$\bC$ on~$\mc A$ through this
adjunction, there is a  
canonical distributive law 
$ \chi \colon \bT \bS
\Rightarrow \bS \bT$ (see 
Section~\ref{dist-via-adj});
in the case of an
Eilenberg--Moore adjunction, 
this yields a
bijective correspondence
between lifts of comonads and
distributive laws \cite[Theorem
2.4]{Wol73}.  
In this setup,
right~$\chi$-coalgebra
structures $\rho$ on a functor 
$\rM \co \mc Y \to
\mc B$ correspond 
bijectively to 
$\bC$-coalgebra structures on 
$\rU\rM$~\cite[Proposition~3.5]{KKS15},
which are much easier to describe in
the most important examples of the
theory (see
Section~\ref{exfrombialgs}
for the explicit 
example obtained from 
a bialgebra $H$). 
The aim of the present article is to
provide an analogous description of
left $\chi$-coalgebra structures on
a functor $\rN \co \mc B \to
\mc Z$.

\subsection{Main result}
The classification of left 
$ \chi $-coalgebras 
is obtained by expanding 
\cite[Remark~3.6]{KKS15}, where it
is suggested that in the presence of
an adjunction $\rV \co \mc C
\rightleftarrows \mc B \co \rG$ for
$\bS$ and a comonad $\bQ$ which
extends  $\bT$ through $\rV \dashv
\rG$,  left $\chi$-coalgebra
structures on $\rN$ should 
correspond
bijectively to $\bQ$-opcoalgebra
structures on $\rN \rV$. We depict the resulting setup in the 
following diagram (see Section~\ref{dist-via-adj} for the
more details):  
 \begin{equation*} 
 \begin{gathered}
\xymatrix@C=8mm@R=5mm{
 & & \mc C \ar@/^-3mm/[dd]_\rV    \ar@{}[rrdd]^(.25){}="a"^(.75){}="b" \ar@{=>}^-{\tilde \Omega} "a";"b"  \ar[rr]^{\rQ}   \ar@{}[dd]|{\dashv}
 & & \mc C \ar@/^-3mm/[dd]_\rV \ar@{}[dd]|{\dashv} && 
\\\\
       	\mc Y \ar[rr]^{\rM}                 & & \mc B \ar@/^3mm/[dd]^\rU  \ar@{}[dd]|{\dashv} \ar@/^-3mm/[uu]_\rG    \ar@<0.4ex>[rr]^{\rT} \ar@<-0.4ex>[rr]_{\rS}             & & \mc B \ar@/^3mm/[dd]^\rU  \ar@{}[dd]|{\dashv}
\ar@/^-3mm/[uu]_\rG   \ar[rr]^{\rN} && \mc Z.
\\\\
                        & & \mc A  \ar@/^3mm/[uu]^\rF  \ar@{}[rruu]^(.25){}="a"^(.75){}="b" \ar@{=>}_-{\Omega} "a";"b"  \ar[rr]_{\rC}                      & & \mc A 
\ar@/^3mm/[uu]^\rF &&}
\end{gathered}
\end{equation*}

It turns out that \cite[Remark 3.6]{KKS15} is not  true in general, but 
in the most important examples:
extending 
$\bT$ through $\rV \dashv \rG$
yields a comonad distributive law 
$ \psi \colon \bS \bT \Rightarrow 
\bT \bS$, and the claim holds
when $ \psi $ is invertible with~$ \psi^{-1} = \chi$. 
More precisely, we have:

\begin{thm}\label{hauptsatz} 
In the above setting, 
assume that the mate of
the natural isomorphism~$\tilde
\Omega \co \rQ \rG \Rightarrow \rG
\rT$ which implements the extension
$\bQ$ of $\bT$ through the adjunction~$\rV \dashv \rG$ is a colax
isomorphism of comonads. 
Then 
\begin{enumerate}
  \item the induced distributive law $\psi \co \bS\bT \Rightarrow \bT\bS$ is an isomorphism and 
  \item left~$\psi^{-1}$-coalgebra structures~$(\rN, \lambda)$ on~$\mc Z$ correspond 
	bijectively to~$\mathbb{Q}$-opcoalgebra structures on~$\rN\rV$.
\end{enumerate}
\end{thm}

\subsection{Hopf
module theorems and descent categories}
The probably most natural instance
of Theorem~\ref{hauptsatz} is
constructed as follows:
suppose we are initially given a
mixed distributive law $\theta$
between a monad $\bB$ and a comonad~$\bC$ on a category $\mc A$. 
First, we define 
$\mc B = \mc A^\bB$ to be the
category of~$\bB$-algebras. As
mentioned above,  
\cite[Theorem 2.4]{Wol73} 
determines a 
lift~$\bS$ of $\bC$ to $\mc B$. 
Secondly, we define 
$\mc C=\mc B_{\bS}$ to be the
category of coalgebras over this lift
$\bS$. Using the results
in~\cite[Sections 4.3 and 4.9]{W08}
and Lemma~\ref{mate-computation}
below, we finally obtain a
comonad~$\bQ$ on $\mc C$ that
extends $\bT$, provided
that $\theta$ is an isomorphism.

If $\rV \colon \mc C \rightarrow \mc
B$ denotes the forgetful functor,
then functors $\rK \colon \mc A
\rightarrow \mc C$ 
for which the following
diagram (strictly) commutes
\[
	\xymatrix
	@R+1pc@C+1pc
	{&{\mc B = \mc A^{\bB}}   \\
	{\mc A} \ar[rr]_{\rK} 
	\ar[ur]^{\rF}& & 
	\mc C = (\mc A^{\bB})_{\bS} 
	\ar[ul]_{\rV}
}\] 
correspond to comonad morphisms 
$ \kappa \colon \bT \Rightarrow \bS$
(see \cite[Section 2.2]{MW14}). The
latter appears in different contexts
and particular cases under the
names \emph{canonical
map}~\cite{sch90}, \emph{Galois
map}~\cite{SS05}, \emph{fusion
operator}~\cite{BLV11},
\emph{projection formula}~\cite{CG10},
\emph{multiplicative unitary}
or \emph{Kac-Takesaki operator} \cite{woronowicz}, and
supposedly others.
The
fundamental theorem of Hopf modules
\cite[Proposition 1]{LS69} 
can be generalized 
to \cite[Theorem~4.4]{Me08} which
asserts that 
$\rK$ is an equivalence if
and only if $\kappa$ is an
isomorphism of the comonads
$\bT$ and $\bS$, and, in
addition, the functor~$\rF$ is
comonadic. We then have equivalences
$ \mc A \simeq (\mc A^{\bB})_{\bS}
\simeq (\mc A^{\bB})_{\bT}$.
The latter is the \emph{descent
category} of $\bB$~\cite{Ba12}, 
or the \emph{category of descent
data}~\cite{Me06}. 

If, finally, $\rK^{-1}$ is a
quasi-inverse of $\rK$ and 
$\bP $ is the comonad on $\mc A$
obtained by conjugation of $\bQ$
with $\rK$, then  
(see Corollary~\ref{conjugate}) 
left
$\chi$-coalgebras on $\rN$
correspond to $\bP$-opcoalgebras
on~$\rN\rF$. So only in this case, we
indeed have a description of left $
\chi $-coalgebras which is
completely dual to that of right $
\chi $-coalgebras obtained 
in \cite{KKS15}. 

\subsection*{Notation and
conventions} We denote categories by
calligraphic capitals, functors by
Roman capitals, and natural
transformations by small Greek letters.
As in this introduction, blackboard
bold letters are used for
(co)mo\-nads, with the same letter in
Roman font denoting the underlying
endofunctors.
The class of objects of a category $\cC$ is denoted by $\Ob(\cC)$. 
Throughout the paper, the composition of functors and the horizontal composition of natural transformations is denoted by juxtaposition, while we use~$\circ$ for vertical composition of natural transformations.  

\subsection*{Acknowledgements}
U.K.~is supported by the DFG
grant ``Cocommutative
comonoids'' (KR 5036/2-1). It also is a
pleasure to thank Niels
Kowalzig for discussions on
the content of this paper.  

\section{Preliminaries}\label{sec_prelims}
In this section, we review
some prerequisites and fix
notation, which closely
follows the one used
in~\cite{KKS15}:
we recall in particular 
the notions of
(co)monads, (co)algebras over
these, of distributive laws
between these and of
(co)algebras over distributive
laws.
Throughout, $\mc A,\mc B,\mc C,\ldots,\mc Z$ are categories. 

\subsection{(Co)monads} \label{monad-comonad} 
Let $\mc C$ be a category.
Recall that a monad on $\mc C$
is a monoid in the category of
endofunctors on $\mc C$: 

\begin{defn}
A \emph{monad}~$\bB$ on $\mc
C$ is a
triple~$(\rB, \mu^{\bB},
\eta^{\bB})$, where~$\rB \co
\mc C \to \mc C$ is
an endofunctor and 
$\mu^{\bB} \co \rB\rB\Rightarrow
\rB$ and $\eta^{\bB} \co
\id_{\mc C} \Rightarrow \rB$  
are natural transformations
(the \emph{multiplication}
and the \emph{unit} of $\bB$) satisfying the associa\-ti\-vity and unitality conditions
$\mu^{\bB} \circ \rB \mu^{\bB} =
\mu^{\bB} \circ \mu^{\bB} \rB$
and $\mu^{\bB} \circ \rB \eta^{\bB}
= \id_{\rB} = \mu^{\bB} \circ
\eta^{\bB} \rB$.
\end{defn}

Here and elsewhere, we follow
the convention to write $
\mathrm{id} _{\rB}$ for the
identity morphism of an object
in a category (here the
identity natural
transformation of the
endofunctor $\rB$), but when
forming a horizontal composition 
(``whiskering'' a natural
transformation with a functor) 
or take a monoidal product
with an identity morphism,
we just write $\rB$. The
symbol $ \circ $ is reserved
for the vertical composition
of natural transformations. 

A comonad on $\mc C$ is a
monad on the opposite
category $\mc C^\mathrm{op}$:

\begin{defn}
A \emph{comonad}~$\bC$ is an
endofunctor $\rC$ together
with natural transformations
$\Delta^{\bC} \co \rC\Rightarrow \rC
\rC$ and~$\varepsilon^{\bC}
\co \rC \Rightarrow \id_{\mc C}$
that are called \emph{comultiplication}
and \emph{counit} and satisfy
the coassociativity and
counitality conditions~$\rC\Delta^{\bC} \circ \Delta^{\bC}
= \Delta^{\bC} \rC \circ
\Delta^{\bC}$ and
$\rC\varepsilon^{\bC}\circ
\Delta^{\bC} = \id_{\rC} =
\varepsilon^{\bC} \rC \circ
\Delta^{\bC}$.
\end{defn}

As monoids in a monoidal category, 
monads naturally form a category.
However, for what we will discuss,
the category on which a monad acts
needs to be viewed as part of
the monad that may change under a
morphism of monads, and this leads 
to two choices:  

\begin{defn}
A \emph{lax morphism}
from a monad~$\bA$ on~$\mc C$ to
a monad~$\bB$ on~$\mc D$ is a
pair~$(\rG, \sigma)$ consisting of a functor~$\rG \co \mc D \to \mc C$ and a natural transformation~$\sigma
\co \rA \rG \Rightarrow \rG \rB$  such
that~$\sigma \circ \mu^{\bA}\rG  
=  \rG \mu^{\bB} \circ \sigma \rB
\circ \rA\sigma$ and $\sigma \circ   \eta^{\bA} \rG  =
\rG \eta^{\bB}$;
a~\emph{colax morphism} 
from~$\bA$ to~$\bB$ is a pair~$(\rF, \tau)$ consisting of a functor $\rF \co \mc C \to \mc D$ and a natural transformation~$\tau \co \rF \rA  \Rightarrow  \rB \rF$ such that $\tau \circ \rF \mu^{\bA}  =  
\mu^{\bB} \rF \circ \rB \tau  \circ
\tau \rA$ and~$\tau  \circ   \rF
\eta^{\bA}=  \eta^{\bB} \rF$.  
\end{defn}

These two choices are about 
switching from $\cCat$ to 
$\cCat^\mathrm{op}$ (reverse the
direction of functors between
categories). This is entirely
independent from switching 
from a specific category 
$\mc C \in \Ob(\cCat)$ to  
$\mc C^\mathrm{op} \in 
\Ob(\cCat)$ which turns
monads into comonads. So by viewing
comonads as monads on te opposite
categories, there are analogously 
two types of morphisms between
comonads; we spell this out to make
clear which type we call lax and
which one colax, and which symbols
we will typically use:

\begin{defn}\label{def-lax-comonad}
A \emph{lax morphism} of from  
a comonad $\bC$ on $\mc C$ 
to a comonad~$\bD$ on $\mc D$ 
is a pair $(\rG, \sigma )$
consisting of a functor~$\rG \co
\mc D \to \mc C$ and a natural
transformation~$\sigma \co \rC \rG \Rightarrow \rG \rD$ such that 
$\rG \Delta^{\bD} \circ \sigma
= \sigma \rD \circ \rC \sigma
\circ \Delta^{\bC} \rG$ and  
$\rG \varepsilon^{\bD} \circ
\sigma = \varepsilon^{\bC}
\rG$; a \emph{colax morphism}
from~$\bC$ to~$\bD$ is a pair~$(\rF, \sigma)$, where~$\rF \co \mc C \to \mc
D$ is a functor and~$\tau \co \rF
\rC  \Rightarrow  \rD \rF$ is a natural
transformation such that  
$\Delta^{\bD} \rF \circ \tau =
\rD \tau \circ  \tau \rC \circ
\rF \Delta^{\bC}$ and
$\varepsilon^{\bD}  \rF \circ
\tau = \rF \varepsilon^{\bC}$. 
\end{defn}

\begin{rem} \label{rem-inv} It follows from
these definitions that if~$\rG \co \mc D \to \mc C$ is a functor and if~$\sigma \co \rC \rG \Rightarrow \rG \rD$ is an invertible natural transformation such that~$(\rG,\sigma)$ is a lax morphism of comonads from $\bC$ to $\bD$, then $(\rG, \sigma^{-1})$ is a colax morphism of comonads from $\bD$ to $\bC$.\
\end{rem}

\subsection{(Co)algebras over
(co)monads}
\label{algebras}
Let $\bB$ be a monad on $\mc
C$. 

\begin{defn}
A \emph{$\bB$-algebra} on a
category $\mc X$ is a pair 
$(\rM, \alpha)$ consisting of
a functor $\rM \co \mc X \to
\mc C$ and a natural
transformation~$\alpha \co \rB
\rM \Rightarrow \rM$ (called
\emph{action}) such that 
$\alpha \circ \mu^{\bB} \rM    =
\alpha \circ \rB \alpha$ and 
$\alpha \circ  \eta^{\bB} \rM=
\id_{\rM}$.
\end{defn}

When it comes to the notion of
morphism between algebras, 
$\mc X$ and $\bB$ will always be
fixed in this paper:

\begin{defn}
A morphism between 
$\bB$-algebras~$(\rM, \alpha)$
and~$(\rN, \beta)$ on $\mc X$ is a natural transformation~$f\co \rM \Rightarrow \rN$ such that~$f \circ \alpha = \beta \circ  \rB f$. 
The \emph{Eilenberg--Moore
category of}~$\bB$ is the
category~$\mc C^{\bB}$ of
$\bB$-algebras on the terminal
category.
\end{defn}

Thus the objects of $\mc
C^\bB$ are pairs~$(X,\alpha)$
of an object~$X \in \Ob(\mc
C)$ and a morphism~$\alpha \co
\rB(X)\to X$ in $\mc C$ such
that $\alpha \mu^{\bB}_X=
\alpha \rB(\alpha)$
and~$\alpha \eta^{\bB}_X =
\id_X$; 
here
the concatenation
denotes the composition of 
morphisms in $\mc C$ (recall 
we only use $ \circ $ to
denote the
vertical composition of
natural transformations). 
A morphism from~$(X,\alpha)$ to~$(Y, \beta)$ is a morphism~$f \co X\to Y$ in $\mc C$ such that~$f\alpha= \beta \rB(f)$.

As with all other notions, 
these definitions have the
obvious dual version:

\begin{defn}
Let $\bC$ be a comonad on $\mc C$. 
A \emph{$\bC$-coalgebra} on a
category~$\mc X$ is a
pair~$(\rM, \gamma)$,
where~$\rM \co \mc X \to \mc
C$ is a functor and~$\gamma
\co \rM \Rightarrow \rC \rM$ is a
natural transformation (called
\emph{coaction}) which
satisfies
$\Delta^{\bC} \rM \circ \gamma= \rC
\gamma  \circ \gamma$ and
$\varepsilon^{\bC} \rM  \circ \gamma
= \id_{\rM}$.
A morphism between coalgebras~$(\rM, \gamma)$ and~$(\rN,\delta)$ is a natural transformation~$f\co \rM \Rightarrow \rN$ such that~$ \delta \circ f =  \rC f \circ \gamma$. 
If the cate\-gory~$\mc X$ is
the terminal category, then
the category formed by
coalgebras over~$\bC$ is
called the
\emph{Eilenberg--Moore
category} of~$\bC$ and is denoted by~$\mc C_{\bC}$. 
\end{defn} 
So an object in $\mc C_{\bC}$
consists of an object~$X \in \Ob(\mc
C)$ together with
a morphism~$\gamma
\co X\to \rC(X)$ in $\mc C$
such that~$\Delta^{\bC}_X
\gamma = \rC(\gamma) \gamma$
and~$\epsilon^{\bC}_X \gamma =
\id_X$, and a morphism
from~$(X,\gamma)$ to~$(Y,
\delta)$ is a morphism~$f \co
X\to Y$  in $\mc C$ such
that~$\delta f= \rC(f)\gamma
$. 

\begin{rem}
In \cite{KKS15}, 
$\mc C_\bC$ was called the
\emph{coKleisli category} of
$\bC$. 
\end{rem}

%\subsection{Opcoalgebras over
%a comonad}
%\label{opcoalgebras} 
In Section \ref{mainresult}, we
will furthermore use the
concept of an opcoalgebra over
a comonad $\bC$ on $\mc C$ (one similarly
defines opalgebras over a
monad, but we we will not use
this term):

\begin{defn}
A \emph{$\bC$-opcoalgebra} on
$\mc X$ is 
a functor $\rM \co \mc C \to
\mc X$ together with a natural
transformation~$\gamma \co \rM \Rightarrow
\rM \rC$ satisfying
$\rM \Delta^{\bC} \circ \gamma =
\gamma \rC \circ \gamma$ and $\rM
\varepsilon^{\bC}  \circ \gamma =
\id_{\rM}$. 
A morphism between 
opcoalgebras~$(\rM, \gamma)$
and~$(\rN,\delta)$ on $\mc X$ is a natural transformation~$f\co \rM \Rightarrow \rN$ such that~$ \delta \circ f = f  \rC \circ \gamma $. 
\end{defn}

\begin{rem}
Some authors call (co)algebras
\emph{left (co)modules} and 
op(co)al\-ge\-bras \emph{right
(co)modules}. This is  
motivated by the 
examples of (co)monads on 
monoidal categories that are 
given by tensoring with
(co)monoids therein. 
However, we stick
to the original terminology 
which is more suggestive in
examples in which
the (co)algebras are algebraic
structures such as
semigroups, rings etc. 
\end{rem}

\subsection{From (co)monads to adjunctions and back} \label{monads-and-adjunctions}
We recall:

\begin{defn}
An \emph{adjunction} is a pair of functors
$\rF\co \mc A \rightleftarrows \mc B
\co \rU$,
together with natural
transformations~$\eta \co \id_{\mc
A} \Rightarrow \rU\rF$ and $\varepsilon \co
\rF\rU \Rightarrow \id_{\mc B}$ such that 
\begin{equation}\label{snake-relations}
\varepsilon \rF  \circ \rF\eta = \id_{\rF} \quad \text{and} \quad \rU \varepsilon \circ \eta \rU = \id_{\rU}.
\end{equation}
 We call~$\rF$ the \emph{left adjoint }and~$\rU$ the \emph{right adjoint}, which we shortly denote~$\rF \dashv \rU$.
The natural transformations~$\eta$
and~$\varepsilon$ are respectively
called the \emph{unit} and  the
\emph{counit} of the adjunction~$\rF
\dashv \rU$.
\end{defn}

In this situation, the
endofunctor~$\rB = \rU \rF$ on~$\mc
A$ is part of the monad $\bB= (\rB,
\mu^{\bB}, \eta^{\bB})$,
where~$\mu^\bB = \rU \epsilon \rF$
and~$  \eta^\bB = \eta$, and 
the endofunctor~$\rT = \rF \rU$ on~$\mc B$ is part of the comonad~$\bT = (\rT, \Delta^{\bT}, \varepsilon^{\bT})$, where~$\Delta^{\bT} = \rF\eta \rU $ and $\varepsilon^{\bT} = \varepsilon.$ 

Conversely, any monad~$\bB$ on~$\mc
A$ induces the~\emph{Eilenberg--Moore
adjunction}~$\rF\co \mc A
\rightleftarrows \mc A^{\bB} \co
\rU$. Here, the~\emph{free algebra
functor}~$\rF$ is given by setting~$\rF(X) =
(\rB(X),
\mu^{\bB}_X)$ on objects
and~$\rF(f)= \rB(f)$ on morphisms,
while~$\rU$ is the forgetful
functor defined by~$\rU(X,\alpha) = X$ on objects and by~$\rU(f)= f$ on morphisms.
The unit~$\eta$ of the adjunction is given by setting~$\eta_X= \eta^{\bB}_X$ for~$X \in \Ob( \mc A)$. 
The counit~$\varepsilon$ of the adjunction is defined by setting~$\varepsilon_{(X,\alpha)} = \alpha$ for~$(X,\alpha) \in \Ob(\mc A^{\bB})$.
For a comonad~$\bS$ on~$\mc
B$, the Eilenberg--Moore
adjunction~$\rV\co \mc  B_{\bS}
\rightleftarrows \mc  B \co \rG$
is formed by the  
\emph{cofree coalgebra
functor}, which is 
given by~$\rG(X) = (\rS(X),
\Delta^{\bS}_X)$ on objects
and by~$\rG(f)= \rS(f)$ on morphisms,
while~$\rV$ is the forgetful
functor. 
The unit~$\eta'$ of the adjunction
is defined by
setting~$\eta'_{(X,\gamma)}= \gamma$
for~$(X,\gamma) \in \Ob(\mc B_{\bS})$. 
The counit~$\varepsilon'$ of the
adjunction is given by
setting~$\varepsilon'_X =
\varepsilon^{\bS}_X$ for ~$X \in
\Ob(\mc B)$.

\subsection{Distributive
laws}\label{distlawsandcoalgs}
Let
$\bB$ be a monad and 
$\bC,\bD$ be comonads on $\mc
C$.

\begin{defn}
A \emph{distributive law} of~$\bC$ over~$\bD$ is a natural transformation~$\chi \co \rC \rD \Rightarrow \rD\rC$ such that~$(\rD,\chi)$ is a lax endomorphism of~$\bC$ and~$(\rC, \chi)$ is a colax endomorphism of~$\bD$.
A \emph{mixed distributive} law of ~$\bB$ over ~$\bC$ is a natu\-ral transformation~$\theta \co \rB \rC \Rightarrow \rC \rB$ such that~$(\rC,\theta)$ is a lax endomorphism of the monad~$\bB$ and~$(\rB, \theta)$ is a colax endomorphism of the comonad~$\bC$.
\end{defn}

\begin{rem}
Some authors call (mixed)
distributive laws
\emph{entwinings}. 
\end{rem}

Finally, we recall the definition of 
a coalgebra over a distributive law. 
For this, let~$\bT$ and~$\bS$ be
comonads on a category~$\mc B$
and~$\chi$ be a distributive law of~$\bT$ over~$\bS$.

\begin{defn}\label{chicoalgdef}
 A \emph{right~$\chi$-coalgebra} on a category $\mc Y$ is a pair~$(\rM, \rho)$, where ~$\rM \co \mc Y \to \mc B$ is a 
functor and~$\rho \co \rT \rM \Rightarrow
\rS \rM$ is a natural transformation
such that 
$\Delta^{\bS}\rM  \circ \rho = \rS
\rho \circ \chi \rM \circ \rT \rho
\circ \Delta^{\bT} \rM$ and 
$\varepsilon^{\bS} \rM \circ \rho =
\varepsilon^{\bT} \rM$.
A \emph{left~$\chi$-coalgebra}  on a category  $\mc Z$  is a pair~$(\rN,\lambda)$, where $\rN \co \mc B \to \mc Z$ is a 
functor and~$\lambda \co \rN\rS \Rightarrow  \rN\rT$ is a natural transformation such that 
$\rN \Delta^{\bT}  \circ \lambda =
\lambda \rT \circ  \rN  \chi \circ
\lambda \rS \circ  \rN \Delta^{\bS}$
and $\rN \varepsilon^{\bT} \circ
\lambda = \rN \varepsilon^{\bS}$.
\end{defn}

\subsection{Graphical
calculus}\label{graphical-calc}
Following~\cite[Section 1.A.]{BS08},
we use a graphi\-cal calculus to
visualise natural transformations
and equations between them. 
The dia\-grams are made of arcs colored by functors and of boxes, colored by natural transformations between functors.
Arcs colored by identity functors are omitted. We read our diagrams from right to left and from top to bottom: 
let~$\mc C, \mc D$ and~$\mc E$ be categories and let~$\rF\co \mc C \to \mc D$ and~$\rG\co \mc D\to \mc E$ be functors. The natural transformations~$\id_{\rF}$,~$\alpha \co  \rF\Rightarrow \rG$ and its vertical composition~$\beta \circ \alpha$ with a natural transformation~$\beta \co \rG \Rightarrow \rH$ are depicted by
\[
\,
\psfrag{Y}[cc][cc][0.8]{$\rF$}
\rsdraw{0.45}{1}{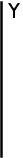}
\;, \qquad
\,
\psfrag{X}[cc][cc][0.8]{$\rG$}
\psfrag{Y}[cc][cc][0.8]{$\rF$}
\psfrag{f}[cc][cc][0.8]{$\alpha$}
\rsdraw{0.45}{1}{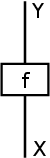}
\;,\qquad 
\,
\psfrag{X}[cc][cc][0.8]{$\rH$}
\psfrag{Y}[cc][cc][0.8]{$\rG$}
\psfrag{Z}[cc][cc][0.8]{$\rF$}
\psfrag{g}[cc][cc][0.8]{$\alpha$}
\psfrag{f}[cc][cc][0.8]{$\beta$}
\rsdraw{0.45}{1}{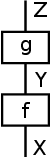}
\;.\]
The horizontal composition~$\beta \alpha$ of natural transformations~$\alpha \co \rF \Rightarrow \rG$ and $\beta \co \rH \Rightarrow \rK$ is represented by placing a picture of~$\alpha$ to the right of the picture of~$\beta$:
\[
\beta \alpha
=
\,
\psfrag{X}[cc][cc][0.8]{$\rK$}
\psfrag{Y}[cc][cc][0.8]{$\rH$}
\psfrag{U}[cc][cc][0.8]{$\rG$}
\psfrag{V}[cc][cc][0.8]{$\rF$}
\psfrag{f}[cc][cc][0.8]{$\beta$}
\psfrag{g}[cc][cc][0.8]{$\alpha$}
\rsdraw{0.45}{1}{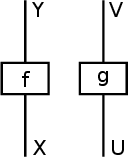}
\;.
\]
The natural transformation depicted
by a diagram depends only on the isotopy class of the diagram. For example, the following equality of diagrams 
\[
\, 
\psfrag{X}[cc][cc][0.8]{$\rK$}
\psfrag{Y}[cc][cc][0.8]{$\rH$}
\psfrag{U}[cc][cc][0.8]{$\rG$}
\psfrag{V}[cc][cc][0.8]{$\rF$}
\psfrag{f}[cc][cc][0.8]{$\beta$}
\psfrag{g}[cc][cc][0.8]{$\alpha$}
\rsdraw{0.45}{1}{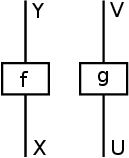}\;= \, 
\psfrag{X}[cc][cc][0.8]{$\rK$}
\psfrag{Y}[cc][cc][0.8]{$\rH$}
\psfrag{U}[cc][cc][0.8]{$\rG$}
\psfrag{V}[cc][cc][0.8]{$\rF$}
\psfrag{f}[cc][cc][0.8]{$\beta$}
\psfrag{g}[cc][cc][0.8]{$\alpha$}
\rsdraw{0.45}{1}{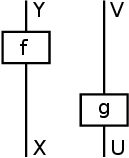}\;=\, 
\psfrag{X}[cc][cc][0.8]{$\rK$}
\psfrag{Y}[cc][cc][0.8]{$\rH$}
\psfrag{U}[cc][cc][0.8]{$\rG$}
\psfrag{V}[cc][cc][0.8]{$\rF$}
\psfrag{f}[cc][cc][0.8]{$\beta$}
\psfrag{g}[cc][cc][0.8]{$\alpha$}
\rsdraw{0.45}{1}{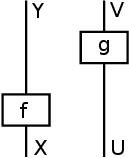}\; \]
reflects the exchange law of
vertical and horizontal composition: \[\beta \alpha = \rK \alpha \circ \beta  \rF =\beta \rG \circ \rH\alpha.\]
Multiplication and unit of a monad $\bB$ on a category are denoted by \[
\,
	\psfrag{O}[cc][cc][0.8]{$\rB$}
	\rsdraw{0.45}{1}{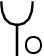}
\; \quad \text{and} \quad
\,
	\psfrag{O}[cc][cc][0.8]{$\rB$}
	\rsdraw{0.45}{1}{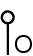}
\;.\]  
Similarly, the comultiplication and counit of a comonad $\bC$ are denoted by
\[
\,
	\psfrag{O}[cc][cc][0.8]{$\rC$}
	\rsdraw{0.45}{1}{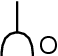}
\; \quad \text{and} \quad
\,
	\psfrag{O}[cc][cc][0.8]{$\rC$}
	\rsdraw{0.45}{1}{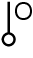}
\;.\]

The unit and the counit of an
adjunction $\rF \dashv \rU$
will be denoted by 
\[ \eta= \,
	\psfrag{F}[cc][cc][0.8]{$\rF$}
	\psfrag{U}[cc][cc][0.8]{$\rU$}
	\rsdraw{0.45}{1}{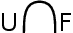}
\;  \quad \text{and}\quad  \varepsilon = 
\,
	\psfrag{F}[cc][cc][0.8]{$\rF$}
	\psfrag{U}[cc][cc][0.8]{$\rU$}
	\rsdraw{0.45}{1}{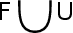}
\; , \] 
so the identities in 
\eqref{snake-relations}
become pictorially 
the \emph{snake relations}
\[\,
	\psfrag{F}[cc][cc][0.8]{$\rF$}
	\psfrag{U}[cc][cc][0.8]{$\rU$}
	\rsdraw{0.45}{1}{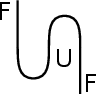}
\;= \,
	\psfrag{U}[cc][cc][0.8]{$\rF$}
	\rsdraw{0.45}{1}{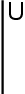}
\; \quad \text{and} \quad \,
	\psfrag{F}[cc][cc][0.8]{$\rU$}
	\psfrag{U}[cc][cc][0.8]{$\rF$}
	\rsdraw{0.45}{1}{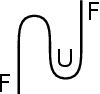}
\;= \,
	\psfrag{U}[cc][cc][0.8]{$\rU$}
	\rsdraw{0.45}{1}{snake2}
\;.\] 

\subsection{Distributive laws from adjunctions} \label{dist-via-adj}
%\subsection{Lifts, extensions, and mates}
Suppose that $\rF \co \mathcal{A}
\rightleftarrows \mathcal{B} \co
\rU$ is an adjunction with a unit
$\eta$ and counit $\varepsilon$.

\begin{defn}\label{lift}
A \emph{lift} of an endofunctor~$\rC$ on~$\cA$ \emph{through the adjunction} $\rF \dashv \rU$
is an endofunctor 
$\rS$ on $\cB$ together with 
a natural isomorphism $\Omega \colon
\rC\rU \Rightarrow \rU\rS$. 
In this situation, we also call   
$\rC$ an \emph{extension} of
$\rS$ \emph{through the adjunction},
and we say that the diagram
%	\[
%	\xymatrix{\cB \ar[d]_\rU \ar[r]^\rS & \cB \ar[d]^\rU\\
%\cA \ar@{=>}[ru]^{\Omega}\ar[r]_\rC & \cA} 		
%	\] I found it too long so I created the new one
\[
\xymatrix{\cB \ar[r]^ \rS  \ar[d]_{\rU} &   \cB \ar[d]^{\rU}\\
	\cA \ar@{}[ru]^(.25){}="a"^(.75){}="b" \ar@{=>}^-{\Omega} "a";"b"  \ar[r]_{\rC}  & \cA}
\]
strongly commutes up to the 
	isomorphism $\Omega$.

We recall now that the \emph{mate} of natural transformation $\Omega \colon \rC\rU \Rightarrow \rU\rS$ 
under the adjunction $\rF \dashv \rU$ is the natural transformation
$\Lambda \colon \rF\rC \Rightarrow \rS\rF$ defined by
\begin{equation} \label{mate}\Lambda =  \varepsilon \rS\rF \circ \rF \Omega \rF \circ \rF\rC \eta =
\,
	\psfrag{O}[cc][cc][0.8]{$\Omega$}
	\psfrag{U}[cc][cc][0.8]{$\rU$}
	\psfrag{C}[cc][cc][0.8]{$\rC$}
    \psfrag{F}[cc][cc][0.8]{$\rF$}
    \psfrag{S}[cc][cc][0.8]{$\rS$}
	\rsdraw{0.45}{1}{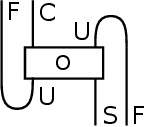}
\;.\end{equation}
\end{defn}
We say that the diagram
\[
\xymatrix{\cB \ar[r]^ \rS  &   \cB \\
	\cA \ar[u]^{\rF}   \ar[r]_{\rC}  & \cA \ar[u]_{\rF} \ar@{}[lu]^(.25){}="a"^(.75){}="b" \ar@{=>}_-{\Lambda} "a";"b" }
\]
	commutes up to $2$-cell $\Lambda$.

\begin{rem} \label{rem-mates}  Assume additionally to the above setup that $\rC$ and $\rS$ are part of comonads $\bS$ and $\bC$.
Note that it follows
from the definitions that $(\rU,\Omega)$ is a lax morphism of comonads from $\bC$ to $\bS$ if and only if $(\rF, \Lambda)$ is a 
colax morphism of como\-nads from $\bC$ to $\bS$. \end{rem}

\begin{defn} \label{comonad-lift} A \emph{lift} of a comonad
$\bC$ on $\cA$ \emph{through the adjunction} $\rF \dashv \rU$
is a comonad 
$\bS$ on $\cB$ together with 
a natural isomorphism $\Omega \colon
\rC\rU \Rightarrow \rU\rS$ such that $(\rU, \Omega)$ is a lax isomorphism of comonads from $\bC$ to $\bS$.
In this situation, we also call   
$\bC$ an \emph{extension} of
$\bS$ \emph{through the adjunction}.  
\end{defn}

Suppose that $\rF \co \mathcal{A} \rightleftarrows \mathcal{B} \co \rU$ is an adjunction and let $\bT$  be the induced comonad on $\mathcal B$ and $\bB$ be the induced monad on $\mathcal A$. 
Further, let $\bS$ be a comonad on $\cB$ which is a lift through the adjunction of a comonad $\bC$ on~$\cA$. 
Assuming that the lift is implemented by the lax isomorphism of como\-nads~$\Omega$ with mate $\Lambda$, we have that the natural transformation
\[ \theta = \Omega^{-1}\rF \circ \rU \Lambda =  
 \,
	\psfrag{O}[cc][cc][0.8]{$\Omega^{-1}$}
	\psfrag{U}[cc][cc][0.8]{$\rU$}
	\psfrag{C}[cc][cc][0.8]{$\rC$}
    \psfrag{F}[cc][cc][0.8]{$\rF$}
    \psfrag{S}[cc][cc][0.8]{$\rS$}
    \psfrag{I}[cc][cc][0.8]{$\Lambda$}
	\rsdraw{0.45}{1}{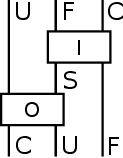}
\;  \co \rB\rC \Rightarrow \rC\rB
\]
is a mixed distributive law of $\bB$ over $\bC$ and the natural transformation
\[
\chi= \Lambda \rU \circ \rF \Omega^{-1} =  
\,
	\psfrag{I}[cc][cc][0.8]{$\Omega^{-1}$}
	\psfrag{F}[cc][cc][0.8]{$\rU$}
	\psfrag{S}[cc][cc][0.8]{$\rC$}
    \psfrag{U}[cc][cc][0.8]{$\rF$}
    \psfrag{C}[cc][cc][0.8]{$\rS$}
    \psfrag{O}[cc][cc][0.8]{$\Lambda$}
	\rsdraw{0.45}{1}{mixeddist}
\; \co \rT\rS \Rightarrow \rS\rT
\] is a  distributive
law  of $\bT$ over $\bS$. See \cite[Theorem 2.9]{KKS15} for details.

\begin{exa}
If~$\mc A=\mc B$ is the category of
finite-dimensional vector spaces 
over a field~$k$, then any vector
space~$V$ gives rise to an adjunction~$\rF \dashv \rU$ with~$\rF = V \otimes -$ and~$\rU = V^* \otimes -$. 
The counit  is induced by the 
evaluation~$V\tens V^* \to k$ and the unit is
 induced by the coevaluation~$k \to V^* \tens V$. 
Taking~$\rS = \rC = \rU$, we 
can consider~$\rS$ as a lift of~$\rC$ by choosing~$ \Omega = \mathrm{id} $. In this
case, the mate~$ \Lambda \colon V \otimes
V^* \otimes - \rightarrow 
V^* \otimes V \otimes -$ 
is the composition of the evaluation~$ V \otimes V^* \rightarrow k$ 
followed by the coevaluation~$k \rightarrow V^* \otimes V$. 
So even when~$\Omega = \mathrm{id}
$, its mate might be not invertible. 
\end{exa}

\section{Main result} \label{mainresult}

In this section, in analogy with a result from \cite{KKS15}, we discuss classifying left $\chi$-coalgebras on $\rN$ (see Theorem~\ref{helpful}). 
We then discuss the motivating exam\-ple from the so called  descent categories. Finally, we apply this to the particular cases from bimonads in the sense of \cite{MW11} and opmonoidal monads (or, bimonads in the sense of \cite{BLV11}).

\subsection{Setup} As recalled
in the introduction, the
main ingredients in the
monadic approach to
constructing duplicial objects
are two comonads, a
distributive law of one over
the other, and left and right coalgebras over the distributive law.  
In what follows, we are in the setup as shown in the diagram
 \begin{equation} \label{setup-kks}
 \begin{gathered}
\xymatrix@C=8mm@R=5mm{
	\mc Y \ar[rr]^{\rM} & & \mc B \ar@/^3mm/[dd]^\rU \ar@{}[dd]|{\dashv} \ar@<0.4ex>[rr]^{\rT} \ar@<-0.4ex>[rr]_{\rS}
 & & \mc B \ar@/^3mm/[dd]^\rU \ar@{}[dd]|{\dashv}  \ar[rr]^{\rN} && \mc Z
\\\\
                        & & \mc A      \ar@/^3mm/[uu]^\rF  \ar@{}[rruu]^(.25){}="a"^(.75){}="b" \ar@{=>}_-{\Omega} "a";"b"  \ar[rr]_{\rC}                      & & \mc A 
\ar@/^3mm/[uu]^\rF &&}.
\end{gathered}
\end{equation}
Here~$\rM$ and~$\rN$ are functors, $\rF \co \mathcal{A} \rightleftarrows \mathcal{B} \co \rU$ is an adjunction,  $\bT$ is  the  comonad on~$\mathcal B$ induced by $\rF \dashv \rU$, $\bC$ a comonad on $\mathcal{A}$ and $\bS$ a comonad on~$\cB$, which lifts the comonad $\bC$ through the adjunction. 
Let $\Omega  \co \rC \rU \Rightarrow \rU \rS$ be the natu\-ral isomorphism which implements this lift and let $\Lambda \co \rF \rC \Rightarrow \rS \rC$  be its mate. 
We further assume that $(\rU,\Omega)$ is a lax morphism of como\-nads and 
following Section~\ref{dist-via-adj}, we denote the induced distributive law of $\bT$ over~$\bS$ by $\chi$.

\subsection{Coefficients revisited} \label{coeffs}
Following~\cite[Proposition~3.5.1]{KKS15}, right~$\chi$-coalgebras~$(\rM,  \rho)$ on $\cX$ correspond bijectively to
$\bC$-coalgebras~$(\rU \rM, \nabla)$ on~$\cX$ via
\[\rho \mapsto \nabla^{\rho} = 
\,
	\psfrag{O}[cc][cc][0.8]{$\Omega^{-1}$}
	\psfrag{V}[cc][cc][0.8]{$\rF$}
	\psfrag{S}[cc][cc][0.8]{$\rS$}
    \psfrag{G}[cc][cc][0.8]{$\rU$}
    \psfrag{C}[cc][cc][0.8]{$\rC$}
    \psfrag{l}[cc][cc][0.8]{$\rho$}
    \psfrag{N}[cc][cc][0.85]{$\rM$}
	\rsdraw{0.45}{1}{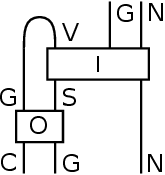}
\; \quad \text{with inverse} \quad \nabla \mapsto \rho^{\nabla} = \,
	\psfrag{O}[cc][cc][0.8]{$\rS$}
	\psfrag{V}[cc][cc][0.8]{$\rU$}
    \psfrag{G}[cc][cc][0.8]{$\rF$}
    \psfrag{T}[cc][cc][0.8]{$\rC$}
    \psfrag{i}[cc][cc][0.8]{$\Omega$}
    \psfrag{N}[cc][cc][0.85]{$\rM$}
    \psfrag{D}[cc][cc][0.85]{$\nabla$}
	\rsdraw{0.45}{1}{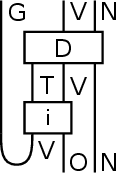}
\;. \]
As suggested in \cite[Remark 3.6]{KKS15}, under additional assumptions, 
classifying left $\chi$-coalgebras $(\rN, \lambda)$ on $\mc Z$ should be analogous.
We have the following result.
\begin{thm} \label{helpful} Additionally to the setting described above, assume that
\begin{enumerate}
  \item $\rV\co \mathcal{C} \rightleftarrows  \mathcal{B} \co \rG$ is an adjunction for the 
	comonad~$\mathbb{S}$,
  \item there is a comonad $\mathbb{Q}$ on~$\mc C$ which extends the comonad
	$\mathbb{T}$ through the adjunction, and
  \item the mate of the natural isomorphism~$\tilde \Omega \co \rQ \rG \Rightarrow \rG \rT$ which implements the extension is a lax isomorphism of comonads. 
\end{enumerate}
Then
\begin{enumerate}
  \item the induced distributive law $\psi \co \bS\bT \Rightarrow \bT\bS$ is an isomorphism and 
  \item left~$\psi^{-1}$-coalgebra structures~$(\rN, \lambda)$ on~$\mc Z$ correspond 
	bijectively to~$\mathbb{Q}$-opcoalgebra structures on~$\rN\rV$.
\end{enumerate}
\end{thm}

The protagonists of our setup and the above theorem are depicted in the following diagram
 \begin{equation} \label{protagonists}
 \begin{gathered}
\xymatrix@C=8mm@R=5mm{
 & & \mc C \ar@/^-3mm/[dd]_\rV    \ar@{}[rrdd]^(.25){}="a"^(.75){}="b" \ar@{=>}^-{\tilde \Omega} "a";"b"  \ar[rr]^{\rQ}   \ar@{}[dd]|{\dashv}
 & & \mc C \ar@/^-3mm/[dd]_\rV \ar@{}[dd]|{\dashv} && 
\\\\
       	\mc Y \ar[rr]^{\rM}                 & & \mc B \ar@/^3mm/[dd]^\rU  \ar@{}[dd]|{\dashv} \ar@/^-3mm/[uu]_\rG    \ar@<0.4ex>[rr]^{\rT} \ar@<-0.4ex>[rr]_{\rS}             & & \mc B \ar@/^3mm/[dd]^\rU  \ar@{}[dd]|{\dashv}
\ar@/^-3mm/[uu]_\rG   \ar[rr]^{\rN} && \mc Z.
\\\\
                        & & \mc A  \ar@/^3mm/[uu]^\rF  \ar@{}[rruu]^(.25){}="a"^(.75){}="b" \ar@{=>}_-{\Omega} "a";"b"  \ar[rr]_{\rC}                      & & \mc A 
\ar@/^3mm/[uu]^\rF &&}
\end{gathered}
\end{equation}

 So analogous to the classification of right $\chi$-coalgebras $(\rM,\rho)$ on  $\cX$ which is recalled above,  and under the conditions of Theorem \ref{helpful}, left~$\psi^{-1}$-coalge\-bras~$(\rN, \lambda)$ on~$\mc Z$ bijectively correspond to $\bQ$-opcoalgebra structures on $\rN\rV$ via
\[\lambda \mapsto \nabla^{\lambda}=  \,
	\psfrag{O}[cc][cc][0.8]{$\rQ$}
	\psfrag{V}[cc][cc][0.8]{$\rV$}
	\psfrag{T}[cc][cc][0.8]{$\rT$}
    \psfrag{G}[cc][cc][0.8]{$\rG$}
    \psfrag{l}[cc][cc][0.8]{$\lambda$}
    \psfrag{i}[cc][cc][0.8]{$\tilde \Lambda^{-1}$}
    \psfrag{N}[cc][cc][0.85]{$\rN$}
	\rsdraw{0.45}{1}{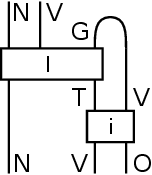}
\; \quad \text{with inverse} \quad \nabla \mapsto \lambda^{\nabla} = \,
	\psfrag{O}[cc][cc][0.8]{$\rQ$}
	\psfrag{V}[cc][cc][0.8]{$\rV$}
	\psfrag{T}[cc][cc][0.8]{$\rT$}
    \psfrag{G}[cc][cc][0.8]{$\rG$}
    \psfrag{l}[cc][cc][0.8]{$\lambda$}
    \psfrag{i}[cc][cc][0.8]{$\tilde \Lambda$}
    \psfrag{N}[cc][cc][0.85]{$\rN$}
    \psfrag{D}[cc][cc][0.85]{$\nabla$}
	\rsdraw{0.45}{1}{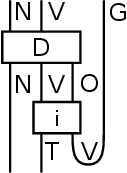}
\;. \]
Here $\tilde \Lambda$ is the uniquely determined mate of $\tilde \Omega$.
We give a detailed proof of Theorem \ref{helpful} in Section \ref{proofmain}. 
An immediate corollary is the following.
\begin{cor} \label{immediate}
 If in addition to the
assumptions of Theorem \emph{\ref{helpful}} the equa\-lity~$\chi =\psi^{-1}$ holds, then left $\chi$-coalgebra structures on~$\rN$ correspond 
	bijectively to~$\mathbb{Q}$-opcoalgebra structures on $\rN\rV$.
\end{cor}

\subsection{The case when $\mc A \simeq \mc C$} \label{xmas-tree} A situation of particular interest is the fol\-lo\-wing. Let~$\rK \co \mc A \to \mc C$ be such that the diagram 
\[\xymatrix@R+1pc@C+1pc{&{\mc B}   \\
	{\mc A} \ar[rr]_{\rK} \ar[ur]^{\rF}& &\mc C\ar[ul]_{\rV}
}\] (strictly) commutes and such that~$\rK$ is an equivalence. 
Denote by~$\bar \rK$ a quasi-inverse, by~$\xi \co \rK \bar \rK \Rightarrow \id_{\mc C}$ and~$\phi \co \id_{\mc A} \Rightarrow \bar \rK \rK$ the counit and the unit of the  associated adjunction~$\rK \dashv \bar \rK$. 
Note that  we also have that $\bar \rK$ is left adjoint to $\rK$ with counit~$\phi^{-1} \co \bar \rK \rK \Rightarrow \id_{\mc A}$ and unit~$\xi^{-1} \co \id_{\mc C} \Rightarrow \rK \bar \rK$.
The conjugate comonad~$\bar{\bQ}$ on $\mc A$ of the comonad~$\bQ$ under the adjunction~$\bar \rK \dashv \rK$ (see~\cite[Definition~2.1.1]{AC13}) is then given by
\[\overline {\bQ} = (\bar{\rK} \rQ \rK, \bar \rK \rQ \xi^{-1} \rQ \rK \circ \bar \rK \Delta^{\mathbb Q} \rK, {\phi}^{-1} \circ \bar \rK \varepsilon^{\mathbb Q}\rK ).\] 
The comultiplication and counit of $\overline {\bQ}$ are graphically depicted as
\[
\Delta^{\overline{\bQ}} = 
\,
	\psfrag{O}[cc][cc][0.8]{$\rQ$}
	\psfrag{R}[cc][cc][0.8]{$\bar{\rK}$}
	\psfrag{K}[cc][cc][0.8]{$\rK$}
    \psfrag{l}[cc][cc][0.8]{$\xi^{-1}$}
	\rsdraw{0.45}{1}{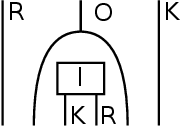}
\;
, \qquad \varepsilon^{\overline{\bQ}} = \,
	\psfrag{O}[cc][cc][0.8]{$\rQ$}
	\psfrag{R}[cc][cc][0.8]{$\bar{\rK}$}
	\psfrag{K}[cc][cc][0.8]{$\rK$}
    \psfrag{l}[cc][cc][0.8]{$\phi^{-1}$}
	\rsdraw{0.45}{1}{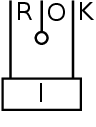}
\;
.\]
We have the following corollary of Theorem \ref{helpful}, which permits an even simpler classification of left $\chi$-coalgebras on $\rN$.  
\begin{cor}
\label{conjugate-general} If
we are in the situation just
described, and if 
 ad\-di\-tionally $\chi =\psi^{-1}$
holds, then left~$\chi$-coalge\-bra structures on $\rN$
correspond to~$\overline \bQ
$-opcoalge\-bra structures
on~$\rN\rF=\rN\rV\rK$.
\end{cor}

\begin{proof}
By Corollary~\ref{immediate}, the left~$\chi$-coalgebra structures on~$\rN $ correspond bijectively to~${\mathbb Q}$-opcoalgebra structures on~$\rN\rV$. 
If~$\nabla  \co \rN\rV \Rightarrow \rN\rV\rQ$ is a~$\bQ$-opcoalgebra structure on~$\rN \rV$, then $\overline \nabla \co \rN\rF \Rightarrow \rN\rF \overline{ \rQ} $, defined by  setting~\[ \overline \nabla= \rN\rV \xi^{-1} \rQ \rK \circ \nabla \rK=\,
	\psfrag{O}[cc][cc][0.8]{$\rQ$}
	\psfrag{R}[cc][cc][0.8]{$\bar{\rK}$}
	\psfrag{K}[cc][cc][0.8]{$\rK$}
    \psfrag{l}[cc][cc][0.8]{$\xi^{-1}$}
    \psfrag{D}[cc][cc][0.8]{\hspace{1.25cm}$\nabla$}
    \psfrag{N}[cc][cc][0.8]{$\rN$}
    \psfrag{V}[cc][cc][0.8]{$\rV$}
	\rsdraw{0.45}{1}{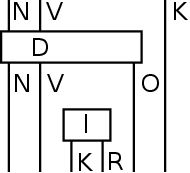}
\;\] is a~$\overline {\mathbb Q}$-opcoalgebra structure on~$\rN \rV \rK = \rN\rF$.
Conversely, if~$\overline \nabla \co \rN \rF \Rightarrow \rN\rF \overline{\rQ}$ is a $ \overline{\bQ}$-opcoalgebra structure on~$\rN \rF $, then~$\nabla \co \rN\rV \Rightarrow \rN\rV \rQ$, defined by setting~\[\nabla = \rN\rV \xi \rQ \xi \circ \overline \nabla \bar \rK \circ \rN\rV \xi^{-1}=\,
	\psfrag{O}[cc][cc][0.8]{$\rQ$}
	\psfrag{R}[cc][cc][0.8]{$\bar{\rK}$}
	\psfrag{K}[cc][cc][0.8]{$\rK$}
    \psfrag{l}[cc][cc][0.8]{$\xi^{-1}$}
    \psfrag{k}[cc][cc][0.8]{\hspace{-0.2cm}$\xi$}
    \psfrag{D}[cc][cc][0.8]{\hspace{1.85cm}$\overline{\nabla}$}
    \psfrag{N}[cc][cc][0.8]{$\rN$}
    \psfrag{V}[cc][cc][0.8]{$\rV$}
	\rsdraw{0.45}{1}{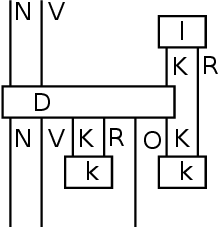}
\;\] is a~$\bQ$-opcoalgebra structure on $\rN \rV$. 
It is a straightforward verification that this correspondence  is bijective.
\end{proof}

\subsection{Proof of Theorem \ref*{helpful}} \label{proofmain}
Since $(\rG, \tilde \Omega)$
is a lax isomorphism of como\-nads from $\bQ$ to $\bT$, it follows from Section \ref{dist-via-adj}, that the distributive law $\psi \co \bS \bT \Rightarrow \bT \bS$ induced by $\tilde \Omega$ is computed by
\[
\xymatrix{\psi \colon
	\rS\rT=\rV\rG\rT \ar[r]^-{\rV \tilde \Omega^{-1}} & \rV\rQ\rG
	\ar[r]^-{\tilde \Lambda \rG} & \rT\rV\rG=\rT\rS},
\] where $\tilde \Lambda \co \rV \rQ \Rightarrow \rT \rV$ is the mate of $\tilde \Omega$ through $\rV \dashv \rG$. 
Since $\tilde \Lambda$ is assumed to be an isomorphism, $\psi $ is also an isomorphism with inverse
\begin{equation}\label{chiinv}
\xymatrix{\psi^{-1} \colon
	\rT \rS=\rT \rV\rG \ar[r]^-{\tilde \Lambda^{-1} \rG} & \rV\rQ\rG
	\ar[r]^-{\rV \tilde \Omega} & \rV \rG\rT=\rS\rT}.
\end{equation}
Before proving the second part of
the statement, we recall from Remark~\ref{rem-mates} that since $(\rG, \tilde \Omega)$ is a lax morphism of comonads from $\bQ$ to $\bT$, $(\rV, \tilde{\Lambda})$  is a colax morphism of comonads from $\bQ$ to $\bT$. In this case, it follows from Remark~\ref{rem-inv} that~$(\rV, \tilde{\Lambda}^{-1})$ is a lax morphism of comonads from $\bT$ to $\bQ$.

Now given a left~$\psi^{-1}$-coalgebra $(\rN,\lambda)$ on $\mc Z$, we define the $\mathbb{Q}$-opcoal\-ge\-bra $(\rN \rV, \nabla^{\lambda})$ on $\mc Z$ by setting  
	\[
	\nabla^\lambda = \rN \tilde \Lambda^{-1}\circ  \lambda \rV  \circ \rN\rV\eta= \,
	\psfrag{O}[cc][cc][0.8]{$\rQ$}
	\psfrag{V}[cc][cc][0.8]{$\rV$}
	\psfrag{T}[cc][cc][0.8]{$\rT$}
    \psfrag{G}[cc][cc][0.8]{$\rG$}
    \psfrag{l}[cc][cc][0.8]{$\lambda$}
    \psfrag{i}[cc][cc][0.8]{$\tilde \Lambda^{-1}$}
    \psfrag{N}[cc][cc][0.85]{$\rN$}
	\rsdraw{0.45}{1}{lambda-to-nabla}
\;.
	\]
We have
\begingroup
\allowdisplaybreaks
\begin{align*}
  ( \rN \rV \Delta^{\bQ})\circ \nabla^{\lambda}  \overset{(i)}{=} &  \,
	\psfrag{O}[cc][cc][0.8]{$\rQ$}
	\psfrag{V}[cc][cc][0.8]{$\rV$}
	\psfrag{T}[cc][cc][0.8]{$\rT$}
    \psfrag{G}[cc][cc][0.8]{$\rG$}
    \psfrag{l}[cc][cc][0.8]{$\lambda$}
    \psfrag{i}[cc][cc][0.8]{$\tilde \Lambda^{-1}$}
    \psfrag{N}[cc][cc][0.85]{$\rN$}
	\rsdraw{0.45}{1}{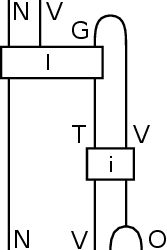}
\; \phantom{ccc}\overset{(ii)}{=}  \phantom{cii} \,
	\psfrag{O}[cc][cc][0.8]{$\rQ$}
	\psfrag{V}[cc][cc][0.8]{$\rV$}
	\psfrag{T}[cc][cc][0.8]{$\rT$}
    \psfrag{G}[cc][cc][0.8]{$\rG$}
    \psfrag{l}[cc][cc][0.8]{$\lambda$}
    \psfrag{i}[cc][cc][0.8]{$\tilde \Lambda^{-1}$}
    \psfrag{N}[cc][cc][0.85]{$\rN$}
	\rsdraw{0.45}{1}{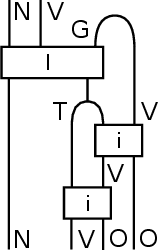}
\; \\ & \\
\overset{(iii)}{=} &  \,
	\psfrag{O}[cc][cc][0.8]{$\rQ$}
	\psfrag{V}[cc][cc][0.8]{$\rV$}
	\psfrag{T}[cc][cc][0.8]{$\rT$}
    \psfrag{G}[cc][cc][0.8]{$\rG$}
    \psfrag{l}[cc][cc][0.8]{$\lambda$}
    \psfrag{i}[cc][cc][0.8]{$\tilde \Lambda^{-1}$}
    \psfrag{N}[cc][cc][0.85]{$\rN$}
    \psfrag{h}[cc][cc][0.85]{$\psi^{-1}$}
    \psfrag{d}[cc][cc][0.85]{\hspace{-1mm}$\id_{\rS}$}
    \psfrag{S}[cc][cc][0.85]{$\rS$}
	\rsdraw{0.45}{1}{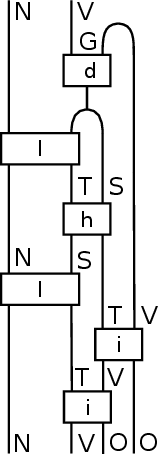}
\; \phantom{ccci} \overset{(iv)}{=} \phantom{cc}
\,
	\psfrag{O}[cc][cc][0.8]{$\rQ$}
	\psfrag{V}[cc][cc][0.8]{$\rV$}
	\psfrag{T}[cc][cc][0.8]{$\rT$}
    \psfrag{G}[cc][cc][0.8]{$\rG$}
    \psfrag{l}[cc][cc][0.8]{$\lambda$}
    \psfrag{i}[cc][cc][0.8]{$\tilde \Lambda^{-1}$}
    \psfrag{N}[cc][cc][0.85]{$\rN$}
    \psfrag{OM}[cc][cc][0.85]{$\tilde \Omega$}
	\rsdraw{0.45}{1}{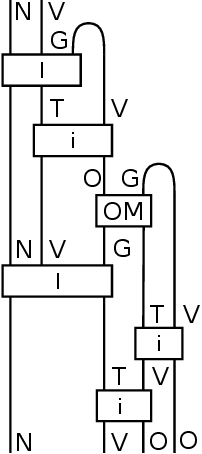}
\; \\ &\\
\overset{(v)}{=}  &  \,
	\psfrag{O}[cc][cc][0.8]{$\rQ$}
	\psfrag{V}[cc][cc][0.8]{$\rV$}
	\psfrag{T}[cc][cc][0.8]{$\rT$}
    \psfrag{G}[cc][cc][0.8]{$\rG$}
    \psfrag{l}[cc][cc][0.8]{$\lambda$}
    \psfrag{i}[cc][cc][0.8]{$\tilde \Lambda^{-1}$}
    \psfrag{N}[cc][cc][0.85]{$\rN$}
    \psfrag{k}[cc][cc][0.85]{$\tilde \Lambda$}
	\rsdraw{0.45}{1}{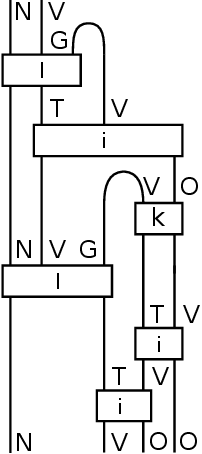}
\; \overset{(vi)}{=} \phantom{ci}
 \,
	\psfrag{O}[cc][cc][0.8]{$\rQ$}
	\psfrag{V}[cc][cc][0.8]{$\rV$}
	\psfrag{T}[cc][cc][0.8]{$\rT$}
    \psfrag{G}[cc][cc][0.8]{$\rG$}
    \psfrag{l}[cc][cc][0.8]{$\lambda$}
    \psfrag{i}[cc][cc][0.8]{$\tilde \Lambda^{-1}$}
    \psfrag{N}[cc][cc][0.85]{$\rN$}
	\rsdraw{0.45}{1}{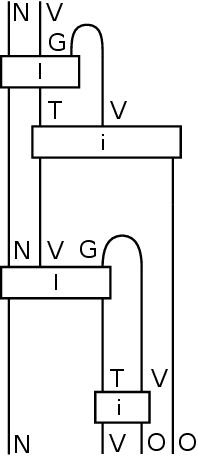},
\;
\end{align*}
\endgroup
which by definition of $\nabla^\lambda$ equals $\nabla^\lambda \rQ \circ \nabla^\lambda$. 
Here $(i)$ follows from the defi\-nition of $\nabla^\lambda$, $(ii)$ follows 
from the fact that $(\rV,\tilde \Lambda^{-1})$ is a lax morphism of comonads from $\bT$ to $\bQ$, $(iii)$ from
the fact that $(\rN, \lambda)$ is a left $\psi^{-1}$-coalgebra on~$\mc Z$, $(iv)$ from the  
computation~\eqref{chiinv} of $\psi^{-1}$ and from the definition of
 comultiplication on~$\bS$ in terms of unit of~$\rV \dashv \rG$,~$(v)$ follows by expres\-sing~$\tilde \Omega$  in terms of its mate~$\tilde \Lambda$ and by using the snake relations~\eqref{snake-relations}, 
 and~$(vi)$ since~$\tilde \Lambda^{-1} \circ \tilde \Lambda = \id_{\rV \rQ}$.
Further, by definition of  $\nabla^\lambda$, the composition  $\rN \rV \circ \nabla^{\lambda}$ equals
\[
\,
	\psfrag{O}[cc][cc][0.8]{$\rQ$}
	\psfrag{V}[cc][cc][0.8]{$\rV$}
	\psfrag{T}[cc][cc][0.8]{$\rT$}
    \psfrag{G}[cc][cc][0.8]{$\rG$}
    \psfrag{l}[cc][cc][0.8]{$\lambda$}
    \psfrag{i}[cc][cc][0.8]{$\tilde \Lambda^{-1}$}
    \psfrag{N}[cc][cc][0.85]{$\rN$}
	\rsdraw{0.45}{1}{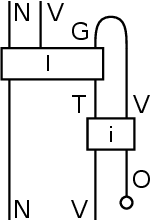}
\;\overset{(i)}{=} \,
	\psfrag{V}[cc][cc][0.8]{$\rV$}
	\psfrag{T}[cc][cc][0.8]{$\rT$}
    \psfrag{G}[cc][cc][0.8]{$\rG$}
    \psfrag{N}[cc][cc][0.85]{$\rN$}
	\rsdraw{0.45}{1}{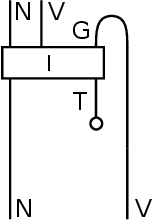}
\;\overset{(ii)}{=}
 \,
	\psfrag{V}[cc][cc][0.8]{$\rV$}
	\psfrag{T}[cc][cc][0.8]{$\rT$}
    \psfrag{G}[cc][cc][0.8]{$\rG$}
    \psfrag{N}[cc][cc][0.85]{$\rN$}
    \psfrag{S}[cc][cc][0.85]{$\rS$}
    \psfrag{d}[cc][cc][0.85]{$\id_{\rS}$}
	\rsdraw{0.45}{1}{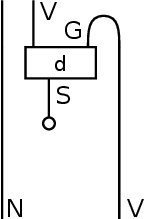}
\;\overset{(iii)}{=}
\id_{\rN \rV}.\]
Here $(i)$ follows from the fact that $(\rV,\tilde \Lambda^{-1})$ is a lax morphism of como\-nads from $\bT$ to $\bQ$, 
$(ii)$ from the fact that $(\rN, \lambda)$ is a left $\psi^{-1}$-coalgebra on~$\mc Z$, and~$(iii)$ follows since the counit of $\bS$ is the counit of adjunction $\rV \dashv \rG$
and by using the snake relations.

	Conversely, to any $\bQ$-opcoalgebra~$(\rN \rV, \nabla)$ on $\mc Z$, we associate a left~$\psi^{-1}$-coal\-ge\-bra $(\rN,  \lambda^{\nabla})$ on $\mc Z$ by setting
	\[
	\lambda^{\nabla} = \rN \rT \varepsilon \circ \rN \tilde \Lambda \rG  \circ \nabla \rG = \,
	\psfrag{O}[cc][cc][0.8]{$\rQ$}
	\psfrag{V}[cc][cc][0.8]{$\rV$}
	\psfrag{T}[cc][cc][0.8]{$\rT$}
    \psfrag{G}[cc][cc][0.8]{$\rG$}
    \psfrag{l}[cc][cc][0.8]{$\lambda$}
    \psfrag{i}[cc][cc][0.8]{$\tilde \Lambda$}
    \psfrag{N}[cc][cc][0.85]{$\rN$}
    \psfrag{D}[cc][cc][0.85]{$\nabla$}
	\rsdraw{0.45}{1}{nabla-to-lambda}
\;.
	\]
We compute
\begingroup
\allowdisplaybreaks
\begin{align*}
  \rN \Delta^{\bT} \circ \lambda^{\nabla} &\overset{(i)}{=} \phantom{o}  \,
	\psfrag{V}[cc][cc][0.8]{$\rV$}
	\psfrag{T}[cc][cc][0.8]{$\rT$}
    \psfrag{G}[cc][cc][0.8]{$\rG$}
    \psfrag{N}[cc][cc][0.85]{$\rN$}
    \psfrag{D}[cc][cc][0.85]{$\nabla$}
    \psfrag{O}[cc][cc][0.85]{$\rQ$}
    \psfrag{i}[cc][cc][0.85]{$\tilde \Lambda$}
	\rsdraw{0.45}{1}{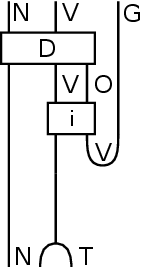}
\;  \phantom{cccci} \overset{(ii)}{=}   \,
	\psfrag{V}[cc][cc][0.8]{$\rV$}
	\psfrag{T}[cc][cc][0.8]{$\rT$}
    \psfrag{G}[cc][cc][0.8]{$\rG$}
    \psfrag{N}[cc][cc][0.85]{$\rN$}
    \psfrag{D}[cc][cc][0.85]{$\nabla$}
    \psfrag{O}[cc][cc][0.85]{$\rQ$}
    \psfrag{i}[cc][cc][0.85]{$\tilde \Lambda$}
	\rsdraw{0.45}{1}{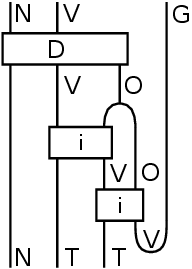}
\;\\ & \\
&\overset{(iii)}{=}  \,
	\psfrag{V}[cc][cc][0.8]{$\rV$}
	\psfrag{T}[cc][cc][0.8]{$\rT$}
    \psfrag{G}[cc][cc][0.8]{$\rG$}
    \psfrag{N}[cc][cc][0.85]{$\rN$}
    \psfrag{D}[cc][cc][0.85]{$\nabla$}
    \psfrag{O}[cc][cc][0.85]{$\rQ$}
    \psfrag{i}[cc][cc][0.85]{$\tilde \Lambda$}
	\rsdraw{0.45}{1}{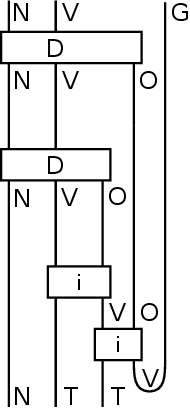}
\; \overset{(iv)}{=}   \,
	\psfrag{V}[cc][cc][0.8]{$\rV$}
	\psfrag{T}[cc][cc][0.8]{$\rT$}
    \psfrag{G}[cc][cc][0.8]{$\rG$}
    \psfrag{N}[cc][cc][0.85]{$\rN$}
    \psfrag{D}[cc][cc][0.85]{$\nabla$}
    \psfrag{O}[cc][cc][0.85]{$\rQ$}
    \psfrag{i}[cc][cc][0.85]{$\tilde \Lambda$}
    \psfrag{k}[cc][cc][0.85]{$\tilde \Lambda^{-1}$}
	\rsdraw{0.45}{1}{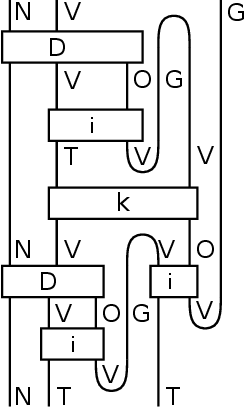}
\; \\ & \\
&\overset{(v)}{=}  \phantom{i} \,
	\psfrag{V}[cc][cc][0.8]{$\rV$}
	\psfrag{T}[cc][cc][0.8]{$\rT$}
    \psfrag{G}[cc][cc][0.8]{$\rG$}
    \psfrag{N}[cc][cc][0.85]{$\rN$}
    \psfrag{l}[cc][cc][0.85]{$\lambda^{\nabla}$}
    \psfrag{O}[cc][cc][0.85]{$\rQ$}
    \psfrag{OM}[cc][cc][0.85]{$\tilde \Omega$}
    \psfrag{k}[cc][cc][0.85]{$\tilde \Lambda^{-1}$}
	\rsdraw{0.45}{1}{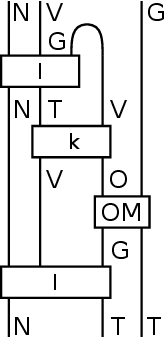}
\; \phantom{cii}\overset{(vi)}{=} 
\,
	\psfrag{T}[cc][cc][0.8]{$\rT$}
    \psfrag{S}[cc][cc][0.8]{$\rS$}
        \psfrag{h}[cc][cc][0.8]{$\psi^{-1}$}
    \psfrag{N}[cc][cc][0.85]{$\rN$}
    \psfrag{l}[cc][cc][0.85]{$\lambda^{\nabla}$}
    \psfrag{O}[cc][cc][0.85]{$\rQ$}
	\rsdraw{0.45}{1}{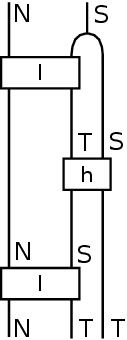}
\;.
\end{align*}
\endgroup
Here~$(i)$ follows from the definition of~$\lambda^{\nabla}$,~$(ii)$ follows from the fact that $(\rV,\tilde \Lambda)$
is a colax morphism of comonads from~$\bQ$ to~$\bT$,~$(iii)$ follows from the assumption that~$(\rN\rV, \nabla)$ is a~$\bQ$-opcoalgebra on~$\mc Z$,~$(iv)$ follows from the
snake relations and~$\tilde{\Lambda}^{-1} \circ \tilde \Lambda = \id_{\rV \rQ}$,~$(v)$ follows from the definition of~$\lambda^{\nabla}$ and by expressing~$\tilde \Omega$ in terms of its mate~$\tilde \Lambda$, and~$(vi)$ follows from the definition of comultiplication on~$\bS$ in terms of the unit of the adjunction $\rV \dashv \rG$ and by an application of equation~\eqref{chiinv}.
Further, we calculate
\[\rN \varepsilon^{\bT} \circ \lambda^{\nabla}\overset{(i)}{=}
\,
	\psfrag{O}[cc][cc][0.8]{$\rQ$}
	\psfrag{V}[cc][cc][0.8]{$\rV$}
	\psfrag{T}[cc][cc][0.8]{$\rT$}
    \psfrag{G}[cc][cc][0.8]{$\rG$}
    \psfrag{l}[cc][cc][0.8]{$\lambda$}
    \psfrag{i}[cc][cc][0.8]{$\tilde \Lambda$}
    \psfrag{N}[cc][cc][0.85]{$\rN$}
    \psfrag{D}[cc][cc][0.85]{$\nabla$}
	\rsdraw{0.45}{1}{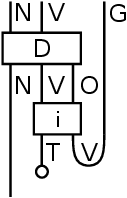}
\;\overset{(ii)}{=} \,
	\psfrag{O}[cc][cc][0.8]{$\rQ$}
	\psfrag{V}[cc][cc][0.8]{$\rV$}
	\psfrag{T}[cc][cc][0.8]{$\rT$}
    \psfrag{G}[cc][cc][0.8]{$\rG$}
    \psfrag{l}[cc][cc][0.8]{$\lambda$}
    \psfrag{i}[cc][cc][0.8]{$\tilde \Lambda$}
    \psfrag{N}[cc][cc][0.85]{$\rN$}
    \psfrag{D}[cc][cc][0.85]{$\nabla$}
	\rsdraw{0.45}{1}{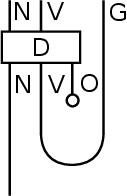}
\; \overset{(iii)}{=}\rN \varepsilon^{\bS}. \]
Here~$(i)$ follows from the
 definition of~$\lambda^{\nabla}$,~$(ii)$ follows from the fact that $(\rV, \tilde \Lambda)$ is a colax morphism of comonads from $\bQ$ to $\bT$, $(iii)$ follows from the fact that~$(\rN\rV, \nabla)$ is a $\bQ$-opcoalgebra on $\mc Z$ and since the counit of $\bS$ is the counit of adjunction $\rV \dashv \rG$.

Finally, the two constructions are in bijective correspondence by inver\-ti\-bility of~$\tilde \Lambda$ and snake relations for adjunction
$\rV\dashv\rG$. \hfill \qed

\subsection{Entwined coalgebras} \label{descent-cats}
Let $\mc A$ be a category, $\bB$ a monad on $\mathcal{A}$, $\bC$ a co\-mo\-nad on $\mc A$, and $\theta$ a mixed distributive law of $\bB$ over $\bC$.
First, consider the Eilenberg--Moore adjunction $\rF \co \mathcal A \rightleftarrows \mathcal A^{\bB} \co \rU$ for $\bB$. 
By definition, we have~$\rB=\rU\rF$, but following Section \ref{monads-and-adjunctions}, we also have that the endo\-functor~$\rF \rU$ on $\mc A^{\bB}$ is part of a comonad $\bT$ on $\mc A^{\bB}$. 
By~\cite{Joh75} (see also~\cite[Section 5.2]{W08}), one defines the comonad $\bC^{\theta}$ on $\mc A^{\bB}$ such that~$\rC \rU=\rU\rC^{\theta}$ as follows.
For any $(X, \alpha)  \in \Ob(\mc A^{\bB})$ and any morphism $f$ in $\mc A^{\bB}$, set
\begin{gather}
 \label{lift-to-EM}\rC^{\theta}(X,\alpha) = (\rC(X), \rC(\alpha_X)\theta_X), \qquad  \rC^{\theta}(f)=\rC(f), \\  
 \label{lift-to-EM2}\Delta^{\bC^{\theta}}_{(X, \alpha)} = \Delta^{\bC}_{X}, \qquad  \varepsilon^{\bC^{\theta}}_{(X, \alpha)}= \varepsilon^{\bC}_{X}. \end{gather}
That is, we can use the
identity natural
transformation~$\Omega \co
\rC \rU \Rightarrow \rU \rC^{\theta}$
to exhibit the comonad lift in the sense of Definition \ref{comonad-lift}. Explicitly,  we set~$\Omega_{(X,\alpha)} = \id_{\rC(X)}$ for all $(X,\alpha) \in \Ob(\mc A^{\bB})$ and we have the following.
\begin{lem}
\label{mate-computation1}The
mate $\Lambda \co \rF\rC \Rightarrow
\rC^{\theta}\rF$ of $\Omega= \id$ is computed by setting for all $X \in \Ob(\mc A)$,
              \[\Lambda_X= \theta_X.\]
              Moreover, the pair $(\rU, \Omega)$ is a lax morphism of comonads from $\bC$ to $\bS$ and  the comonad distributive law $\chi \co \rT \rC^\theta \Rightarrow \rC^{\theta}\rT$
 induced by the adjunction~$\rF\dashv \rU$ is computed by setting for all~$(X,\alpha) \in \Ob(\mc A^{\bB})$,
\begin{equation} \label{computation-dist-law}
\chi_{(X,\alpha)} = \theta_X. 
\end{equation}
\end{lem}
\begin{proof} We first note that for any~$X$ in~$\mc A$, the morphism~$\theta_X$ is a morphism in~$\mc A^{\bB}$ from~$\rF(\rC (X)) $ to~$\rC^{\theta} (\rF(X)) $.  Indeed, this fact follows from the hypo\-thesis that~$(\rC, \theta)$ is a lax endomorphism of~$\bB$. 
Next, it follows from the defi\-nitions, equation~\eqref{mate}, naturality of~$\theta$, and unitality of~$\bB$, that the mate~$\Lambda$ of~$\Omega$ is computed by  
\[\Lambda_X = \varepsilon_{\rC^{\theta}(\rF(X))}\rF(\Omega_{\rF(X)})\rF(\rC(\eta_X))= \rC(\mu^{\bB}_X)\theta_{\rB(X)}\rB(\rC(\eta^{\bB}_X))= \theta_X\] for all~$X \in \Ob(\mc A)$. 
The fact that~$(\rU, \Omega)$  is a lax morphism of comonads from~$\bC$ to~$\bS$ follows since~$\Omega= \id$, by functoriality of~$\rC$, and by equation~\eqref{lift-to-EM2}.
Finally, the distributive law~$\chi \co \rT \rC^\theta \Rightarrow \rC^{\theta}\rT$
 arising from the adjunction~$\rF\dashv \rU$ is computed by setting for all~$(X,\alpha)$ in~$\mc A^{\bB}$, 
\[ 
\chi_{(X,\alpha)} = \Lambda_{\rU(X,\alpha)} \rF(\Omega^{-1}_{(X,\alpha)}) = \theta_X,
\] which proves equation \eqref{computation-dist-law}.
\end{proof}

We next consider the category~$(\mc A^{\bB})_{\bC^{\theta}}$ of coalgebras over the comonad~$\bC^{\theta}$ on $\mc A^{\bB}$, which is called the category of \emph{entwined coalgebras}. 
By Section~\ref{monads-and-adjunctions}, we have the forgetful-cofree adjunction $\rV \co (\mc A^{\bB})_{\bC^{\theta}} \rightleftarrows \mc A^{\bB} \co \rG$, whose associated como\-nad is precisely $\bC^{\theta}$. 
By \cite[Section 4.9]{W08}, one defines a comonad~$\bT^{\theta}$ on~$ (\mc A^{\bB})_{\bC^{\theta}} $ such that~$ \rV \rT^{\theta} = \rT\rV$ as follows. 
For any object~$((X, \alpha), \delta)$ and any morphism~$f$ in~$(\mc A^{\bB})_{\bC^{\theta}}$, set
\begin{gather}
\label{extension}
\rT^{\theta}((X, \alpha), \delta)=  (\rT(X,\alpha), \chi_{(X,\alpha)} \rT(\delta)) = ((\rB(X),\mu^{\bB}_X), \theta_X \rB(\delta)), 
\\  
\label{extension2}\rT^{\theta} (f) = \rT(f), \quad \Delta^{\bT^{\theta}}_{((X, \alpha), \delta)} = \Delta^{\bT}_{(X, \alpha)}, \quad  \varepsilon^{\bT^{\theta}}_{((X, \alpha), \delta)}= \varepsilon^{\bT}_{(X, \alpha)}.  \end{gather}
\begin{rem}
Beware that $\rT^{\theta}$ is not to $\rT$, what $\rC^{\theta}$ to $\rC$ is. 
While the comonad ~$\bT^{\theta}$ extends the comonad $\bT$ through $\rV \dashv \rG$, the comonad~$\bC^{\theta}$ lifts~$\bC$ through~$\rF \dashv \rU$ (see Definition \ref{comonad-lift}). 
The choice of this slightly inconsistent notation is motivated by the formulas~\eqref{lift-to-EM} and~\eqref{extension2}. 
Note that~$\bT^{\theta}$ stands for~$\bQ$ and $\bC^{\theta}$ stands for $\bS$ in the more general setup of Section \ref{coeffs}.
\end{rem}

We can use the
identity natural
transformation$~\tilde \Lambda \co \rV \rT^{\theta} \Rightarrow \rT \rV$ to exhibit the comonad extension in the sense of Definition~\ref{comonad-lift}. Explicitly, we set~$\tilde \Lambda_{((X,\alpha), \delta)} = \id_{\rT(X,\alpha)} = \id_{(\rB(X), \mu^{\bB}_X)}$  for all~$((X,\alpha), \delta)$ in~$(\mc A^{\bB})_{\bC^{\theta}}$.
We have the following lemma, which is similar to Lemma \ref{mate-computation1}.

\begin{lem} \label{mate-computation}
      The unique natural transformation $\tilde \Omega \co \rT^{\theta} \rG \Rightarrow \rG\rT$ whose mate equals~$\tilde \Lambda$ is computed by setting~\[\tilde{\Omega}_{(X,\alpha)} = \theta_X\] for all $(X,\alpha)\in \Ob(\mc A^{\bB})$.
In particular, $\tilde{ \Omega }$  implements a comonad extension~$\bT^{\theta}$ of $\bT$ if and only if the morphism~$\theta_X$ is an isomorphism from~$\rT^{\theta}(\rG (X)) $ to~$\rG (\rT(X)) $ for any $X \in \Ob(\mc A)$.

\end{lem} 

\begin{proof}
We first note that for any~$X \in \Ob(\mc A)$, the morphism~$\theta_X$ is a morphism in~$(\mc A^{\bB})_{\bC^{\theta}}$ from~$\rT^{\theta}(\rG (X)) $ to~$\rG (\rT(X)) $.  Indeed, this fact follows from the hypothesis that~$(\rB, \theta)$ is a colax endomorphism of~$\bC$. Next, denote by $\varepsilon'$ and $\eta'$ the counit and unit of the adjunction $\rV\dashv \rG$. 
By equation \eqref{mate} and by snake relations \eqref{snake-relations}, we have for any~$(X, \alpha) \in \Ob(\mc A^{\bB})$, 
 \[\tilde \Omega_{(X,\alpha)} = \rG (\rT (\varepsilon'_{(X,\alpha)}))\rG (\tilde \Lambda_{\rG(X,\alpha)}) \eta'_{\rT^{\theta}(\rG (X, \alpha))}.\]
By definitions and functoriality of $\rG$, it follows that 
\[\tilde{\Omega}_{(X,\alpha)} = \rC (\rB(\varepsilon^{\bC}_X))\theta_{\rC(X)}\rB(\Delta_X^{\bC}),\]
 which by naturality of $\theta$ and counitality of $\bC$ implies the statement.
 \end{proof}
 
The following diagram presents the above described situation:
\begin{equation}  \label{setup-desc-cats} 
\begin{gathered}
\xymatrix@C=7mm@R=5mm{
 & & (\mc A^{\bB})_{\bC^{\theta}} \ar@{}[rrdd]^(.25){}="a"^(.75){}="b" \ar@{=>}^-{\tilde \Omega} "a";"b" \ar@/^-3mm/[dd]_\rV  \ar[rr]^{\rT^{\theta}}   \ar@{}[dd]|{\dashv}
 & & (\mc A^{\bB})_{\bC^{\theta}} \ar@/^-3mm/[dd]_\rV \ar@{}[dd]|{\dashv} && 
\\\\
	\mc Y \ar[rr]^{\rM} & & \mc A^{\bB}  \ar@/^-3mm/[uu]_\rG   \ar@/^3mm/[dd]^\rU \ar@{}[dd]|{\dashv} \ar@<0.4ex>[rr]^{\rT} \ar@<-0.4ex>[rr]_{\rC^{\theta}}
 & &\mc A^{\bB} \ar@/^-3mm/[uu]_\rG  \ar@/^3mm/[dd]^\rU \ar@{}[dd]|{\dashv}  \ar[rr]^{\rN} && \mc Z.
\\\\
                        & & \mc A  \ar@/^3mm/[uu]^\rF \ar@<0.4ex>[rr]^{\rB}  \ar@{}[rruu]^(.25){}="a"^(.75){}="b" \ar@{=>}_-{\Omega} "a";"b"  \ar@<-0.4ex>[rr]_{\rC}                   & & \mc A 
\ar@/^3mm/[uu]^\rF &&}
\end{gathered}
\end{equation}
We have the following corollary of Theorem \ref{helpful}.

 \begin{cor} \label{desc-cat-cor}
If the mixed distributive law $\theta$ is invertible, then the left $\chi$-coalgebra structures on 
$\rN $ correspond bijectively to~$\mathbb{T}^{\theta}$-opcoalgebra structures on $\rN\rV$.
\end{cor}

\begin{proof}
By using Lemma \ref{mate-computation} and the invertibility of $\theta$, we have that the  diagram 
 \[
	\xymatrix{\cA^{\bB} \ar[d]_\rG \ar[r]^\rT & \cA^{\bB} \ar[d]^\rG\\
		(\cA^{ \bB})_{\bC^{\theta}} \ar@{}[ru]^(.25){}="a"^(.75){}="b" \ar@{=>}^-{\tilde \Omega} "a";"b" \ar[r]_{\rT^{\theta}} & (\cA^{ \bB})_{\bC^{\theta}}}
	\]
 commutes up to the natural isomorphism ~$\tilde \Omega$.  
It follows from the definition of~$\Lambda$, functoriality, and equation~\eqref{extension2} that~$(\rV, \tilde \Lambda)$ is a colax morphism of comonads from~$\bT^{\theta}$ to~$\bT$. 
By using Remark \ref{rem-mates}, this implies that $(\rG, \tilde \Omega)$ is a lax morphism of comonads from $\bT^{\theta}$ to $\bT$.
Finally, since~$\tilde \Lambda $ is  an isomorphism, the induced distributive law~$\psi \co \rC^{\theta} \rT \Rightarrow \rT\rC^{\theta}$, which is  computed by~$\psi=\tilde \Lambda \rG \circ \rV \tilde \Omega ^{-1}$  (see Section~\ref{dist-via-adj}),  is an isomorphism. 
 By using Lemma~\ref{mate-computation}, its inverse is computed by setting for any~$(X,\alpha) \in \Ob(\mc A^{\bB})$, \[\psi^{-1}_{(X,\alpha)}= \rV\tilde \Omega_{(X,\alpha)} \tilde \Lambda^{-1}_{\rG(X,\alpha)} = \theta_X.\]
By Lemma~\ref{mate-computation1}, we have that $\chi = \psi^{-1}$. The conclusion follows by application of Corollary \ref{immediate}.
\end{proof}

\subsection{Generalized Hopf module
theorems}
\label{generalised-hopf-thm} Assume now
that we are in the setup of Section~\ref{descent-cats}. 
An even simpler classification of
left~$\chi$-coalgebra structures on
$\rN$ is possible in the case when
the category $\mc A$ is equivalent
to the cate\-gory $(\mc
A^{\bB})_{\bC^{\theta}}$ of entwined
coalgebras. 
To this end, we recall a result from~\cite{Me08} (see also~\cite{MW10} and \cite{MW11}). 
Let~$\rK \co \mc A \to (\mc
A^{\bB})_{\bC^{\theta}} $ be a
functor such that the following
diagram (strictly) commutes
\[\xymatrix@R+1pc@C+1pc{&{\mc A^{\bB}}   \\
	{\mc A} \ar[rr]_{\rK} \ar[ur]^{\rF}& & (\mc A^{\bB})_{\bC^{\theta}}. \ar[ul]_{\rV}
}\] 
In other words, we have for any $X \in \Ob(\mc A)$, 
\[\rK(X)= (\rF(X), \gamma_X), \]  
 where~$\gamma \co \rF \Rightarrow \rC^{\theta}\rF$ is a natural transformation giving the functor $\rF$ a coalgebra structure over  $\rC^{\theta}$. 
The latter induces the \emph{Galois map} $\kappa$ associated to $\rK$ and the adjunction $\rF\dashv \rU$, which is a natural transformation from~$\rT$ to~$\rC^{\theta}$ defined by  \[\kappa = \rC^{\theta}\varepsilon \circ \gamma \rU \co \rT \Rightarrow \rC^{\theta}.\] 
We recall the following result obtained by combining \cite[Theorem 4.4]{Me08} with \cite[Proposition 2.5]{MW14}.
\begin{thm} \label{generalised-hopf-thm-mw} The functor $\rK$ is an equivalence if and only if
\begin{enumerate}
  \item the natural transformation $\kappa \rF$ is an isomorphism and
    \item \label{beck}the functor~$\rF$ is comonadic, that is,  the functor~$\rL\co \mc A \to (\mc A^{\bB})_{\bT}$, which is the unique functor such that $\rV\rL=\rF$ and $\rL\rU=\rG$, is an equivalence.
\end{enumerate}  
\end{thm}

\begin{rem}
Note that the category~$(\mc A^{\bB})_{\bT}$ appears under the name \emph{descent category} in~\cite{Ba12}, or \emph{category of descent data with respect to the monad}~$\bB$ in~\cite{Me06}. 
\end{rem}

If~$\rK$ is an equivalence with a quasi-inverse $\bar \rK$, then $\rK \dashv \bar \rK$. If~$\xi$ is the counit and~$\phi$ the unit of the adjunction, then the conjugate comonad~$\overline {\mathbb{T}^{\theta}}$ on~$\mc A$ of the comonad~$\bT^{\theta}$ on~${(\mc A^{\bB})_{\bC^{\theta}}} $ (see Section \ref{xmas-tree}) is defined by
\[\overline {\mathbb{T}^{\theta}} = (\bar{\rK} \rT^{\theta} \rK, \bar \rK \rT^{\theta}  \xi^{-1} \rT^{\theta} \rK \circ \bar \rK \Delta^{\mathbb{T}^{\theta}} \rK, {\phi}^{-1} \circ \bar \rK \varepsilon^{\mathbb{T}^{\theta}}\rK  ).\] 
By combining Corollary  \ref{desc-cat-cor} with Corollary \ref{conjugate-general}, we obtain the fol\-lo\-wing.
\begin{cor} \label{conjugate} If the mixed distributive law $\theta$ is invertible and if $\rK$ is an equivalence, then (in the above notation) the left $\chi$-coalgebras on $\rN$ correspond to $\overline {\bT^{\theta}}$-opcoalgebras on~$\rN\rF$. 
\end{cor}

\subsection{An example from bimonads} \label{bimonads-mw}

In this section, we apply the theory described in Sections~\ref{descent-cats} and~\ref{generalised-hopf-thm} to bimonads in the sense of~\cite{MW11} and opmonoidal monads (called bimonads in~\cite{BLV11} and Hopf monads in~\cite{Moerdijk}).

\begin{defn} \label{bimonad-mw}
A \emph{bimonad} on $\mc A$ (in the sense of \cite{MW11}) is a sextuple~$\bB=(\rB, \mu^{\bB}, \eta^{\bB}, \Delta^{\bB}, \varepsilon^{\bB}, \theta)$, where $\bB=(\rB,  \mu^{\bB}, \eta^{\bB})$ is a monad,~$\bC=(\rB,\Delta^{\bB}, \varepsilon^{\bB})$ is a comonad and the natural transformation~$\theta \co \bB\bC \Rightarrow \bC\bB$ is a mixed distributive law of~$\bB$ over~$\bC$ such that~$\varepsilon^{\bB} \circ \mu^{\bB} =\varepsilon^{\bB} \circ \varepsilon^{\bB} \rB$,~$\Delta^{\bB} \circ \eta^{\bB} = \eta^{\bB} \rB \circ \eta^{\bB}$,~$\varepsilon^{\bB}  \circ \eta^{\bB} = \id_{\mc A}$, and 
\begin{equation}\label{comp-mixed}
  \Delta^{\bB}\circ \mu^{\bB} = \rB \mu^{\bB} \circ \theta \rB \circ \rB \Delta^{\bB}.
\end{equation} 
\end{defn}

\begin{rem} In particular, any bimonad~$\bB$ on~$\mc A$ in the sense of the above definition satisfies~$\theta\circ \eta^{\bB}\rC= \rC \eta^{\bB}$ and~$\varepsilon^{\bB} \rB \circ \theta= \rB\varepsilon^{\bB}$. Besides, we note that it follows from axioms for $\bB$ that $\Delta^{\bB} = \theta \circ \rB\eta^{\bB}$ and $\mu^{\bB} = \rC\varepsilon^{\bB} \circ \theta$.
\end{rem}

Let $\bB=(\rB, \mu^{\bB}, \eta^{\bB}, \Delta^{\bB}, \varepsilon^{\bB}, \theta)$ be a bimonad in the sense of Definition~\ref{bimonad-mw}. The como\-nad $\bC=(\rB,\Delta^{\bB}, \varepsilon^{\bB})$ then lifts to a comonad~$\bC^{\theta}$ on~$\mc A^{\bB}$ (see equations~\eqref{lift-to-EM} and~\eqref{lift-to-EM2}). If~$\theta$ is invertible, then the comonad~$\bT$ on~$\mc A^{\bB}$ induced by the Eilenberg--Moore adjunction $\rF\dashv \rU$  is a lift of the comonad~$\bT^{\theta}$ on~$(\mc A^{\bB})_{\bC^{\theta}}$ (see equations~\eqref{extension} and~\eqref{extension2}). In this case, according to Corollary~\ref{desc-cat-cor}, the left $\chi$-coalgebras over a functor $\rN \co \mc A^{\bB} \to \mc Z$ correspond bijectively to $\bT^{\theta}$-opcoalgebras over $\rN\rV$. 
Following \cite[Section 4.3]{MW11}, consider a functor~$\rK \co \mc A \to (\mc A^{\bB})_{\bC^{\theta}}$, given by setting for any object $X$ and any morphism~$f$  in~$\mc A$, \[\rK(X)= (\rF(X), \Delta_X^{\bB})= ((\rB(X),\mu^{\bB}_X), \Delta_X^{\bB}) \quad \text{and} \quad \rK(f)= \rF(f).\] 
Note that it follows from equation~\eqref{comp-mixed} that for any~$X \in \Ob( \mc A)$, the morphism~$\Delta_X^{\bB}$ is a morphism from the algebra~$\rF(X)$ to  the algebra~$\rC^{\theta}\rF(X)$. 
Thus the functor~$\rK$ is well-defined and induces a coalgebra $(\rF, \gamma)$ over $\bC^{\theta}$ by setting~$\gamma_X= \Delta_X^{\bB}$ for all~$X \in \Ob(\mc A)$. 
We clearly have $\rV \rK = \rF$ and the theory from Section~\ref{generalised-hopf-thm} applies. 
The Galois map associated to $\rK$ and~$\rF \dashv \rU$ is given by setting for all $(X,\alpha) \in \Ob(\mc A^{\bB})$, \[ \kappa_{(X,\alpha)}= \rB(\alpha)\Delta_X^{\bB} \co (\rB(X), \mu^{\bB}_X) \to (\rB(X), \rB(\alpha)\theta_X).\] 
We recall from Theorem \ref{generalised-hopf-thm-mw} that the functor $\rK$ is an equivalence if and only if the natural transformation $\kappa\rF$ is an isomorphism and if the free functor~$\rF$ is comonadic.  
Finally, if~$\theta$ is invertible and if~$\rK$ is an equivalence, then Corollary~\ref{conjugate} is applicable and one gets a simpler classification of the left~$\chi$-coalgebras on $\rN$.  

\begin{rem} \label{hopf-monad-mw}
We note that a bimonad~$\bB=(\rB, \mu^{\bB}, \eta^{\bB}, \Delta^{\bB}, \varepsilon^{\bB}, \theta)$ on $\mc A$ such that $\kappa \rF$ is a natural isomorphism in the above sense is called a \emph{Hopf monad} in \cite{MW11}. 
This is different than notion of a Hopf monad in the sense of~\cite{BLV11}, which is considered in the next section.
\end{rem}

\subsection{An example from opmonoidal monads} \label{bimonads-blv} In this section, we outline (following~\cite[Section 5]{MW10}) a particular case of the theory described in Sections~\ref{descent-cats} and \ref{generalised-hopf-thm} in the setting of monoidal categories. 
An \emph{opmonoidal monad} (see~\cite{McC02}, or~\cite{Moerdijk, BLV11} under different names) $\bB$ on a monoidal category $(\mc A, \tens, \uu)$ is an opmonoidal functor and a monad on~$\mc A$ such that~$\mu^\bB$ and $\eta^\bB$ are opmonoidal natural transformations. 
This means in particular that the functor~$\rB$ comes with opmonoidal constraints
\begin{align*}
   \rB_2&= \{\rB_2(X,Y) \co \rB(X\tens Y) \to \rB(X)\tens \rB(Y)\}_{X,Y\in \Ob( \mc A)} \quad \text{and} \\
   \rB_0&\co \rB(\uu) \to \uu\ 
\end{align*}
satisfying several axioms. For example, one has for any $X,Y \in \Ob(\mc A)$,
\begin{equation}\label{opmonoidality-mult} \rB_2(X, Y)\mu^{\bB}_{X\tens Y} = (\mu^{\bB}_X \tens \mu^{\bB}_Y)\rB_2(\rB(X),\rB(Y)) \rB(\rB_2(X,Y)).\end{equation}
Following \cite{BLV11}, an opmonoidal monad $\bB$ gives rise to natural transformations
\begin{align*}
  H^l= & \{H^l_{X,Y} \co \rB(X \tens \rB(Y)) \to \rB(X) \tens \rB(Y)\}_{X,Y \in \Ob(\mc A)} \quad \text{and} \\
  H^r= & \{H^r_{X,Y} \co \rB(\rB(X) \tens Y) \to \rB(X) \tens \rB(Y) \}_{X,Y \in \Ob(\mc A)},
\end{align*}
which are called \emph{left} and \emph{right fusion operators}  and are respectively defined by setting for any $X,Y \in \Ob(\mc A)$,
\begin{align*}
  H^l_{X,Y}=& (\id_{\rB(X)} \tens \mu^{\bB}_Y)\rB_2(X,\rB(Y)), \quad \text{and} \\
  H^r_{X,Y}= & (\mu^{\bB}_X \tens \id_{\rB(Y)})\rB_2(\rB(X),Y).
\end{align*}

\begin{defn} A \emph{left} (respectively, a \emph{right}) \emph{pre-Hopf monad} on $\mc A$ is an opmonoidal monad on $\mc A$  such that~$ H^l_{\uu,-}$ (respectively,  $ H^r_{\uu,-}$ ) is an isomorphism. 
A \emph{pre-Hopf monad} on $\mc A$  is a left and a right pre-Hopf monad. A \emph{Hopf monad }on $\mc A$ is an opmonoidal monad on $\mc A$ such that both left and right fusion operators are isomorphisms. 
\end{defn}

For more details and algebraic properties of Hopf monads in the sense of the above definition,  see~\cite{BLV11}.
Now let $\bB$ be an opmonoidal monad on a monoidal category $(\mc A, \tens, \uu)$  and let $\bC$ be a comonad on~$\mc A$ given by setting  for all $X\in \Ob(\mc A)$  and every morphism $f$ in $\mc A$,
\begin{gather*}\rC(X)= \rB(\uu) \tens X, \quad \rC(f) = \id_{\rB(\uu)} \tens f, \\ 
\Delta^{\bC}_X=  \rB_2(\uu,\uu) \tens X, \quad \varepsilon^{\bC}_X= \rB_0 \tens \id_X. \end{gather*}
Following \cite[Section 5]{MW10}, we define a mixed distributive law of $\bB$ over~$\bC$ by setting for $X\in \Ob(\mc A)$, 
\[\theta_X = (\mu^{\bB}_{\uu}\tens \id_{\rB(X)})\rB_2(\rB(\uu),X) \co \rB\rC(X) \to \rC\rB(X).\] 
Using equations \eqref{lift-to-EM} and \eqref{lift-to-EM2} and the above formulas we have that the lift~$\bC^{\theta}$ of $\bC$ to $\cA^{\bB}$ is given by setting for any $(X, \alpha) \in \Ob (\mc A^{\bB})$ and any morphism $f$ in~$\mc A^{\bB}$,
\begin{gather}
\rC^{\theta}(X,\alpha) = (\rB(\uu) \tens X, (\mu^{\bB}_{\uu}\tens\alpha_X)\rB_2(\rB(\uu),X)), 
 \\  \rC^{\theta}(f)=\rC(f), \qquad \Delta^{\bC^{\theta}}_{(X, \alpha)} = \Delta^{\bC}_{X},  \qquad  \varepsilon^{\bC^{\theta}}_{(X, \alpha)}= \varepsilon^{\bC}_{X}. \end{gather}
Remark that~$\rC^{\theta}(X,\alpha)$ is  the monoidal product of $\bB$-algebras~$(\rB(\uu), \mu^{\bB}_\uu)$ and~$(X, \alpha)$ (see~\cite[Proposition 1.4]{Moerdijk}).
From Section \ref{descent-cats} and particu\-larly Lemma \ref{mate-computation1}, we conclude that the natural isomorphism~$\Omega \co \rC\rU \Rightarrow \rU\rC^{\theta}$ implementing the lift~$\bC^{\theta}$ of $\bC$, its mate~$\Lambda \co \rF\rC \Rightarrow \rC^{\theta}\rF$, and the distributive law~$\chi\co \rT \rC^{\theta} \Rightarrow \rC^{\theta}\rT$ induced by~$\Omega$ and $\Lambda$ are given by the formulas 
\begin{gather*}\Omega_{(X,\alpha)}= \id_{\bB(\uu) \tens X}, \quad \Lambda_X= (\mu^{\bB}_{\uu}\tens \id_{\rB(X)})\rB_2(\rB(\uu),X), \\ \chi_{(X,\alpha)}= (\mu^{\bB}_{\uu}\tens \id_{\rB(X)})\rB_2(\rB(\uu),X).\end{gather*}
Following equations~\eqref{extension} and~\eqref{extension2}, the comonad~$\bT^{\theta}$ on~$ (\mc A^{\bB})_{\bC^{\theta}} $ such that $ \rV \rT^{\theta} = \rT\rV$ is given by
\begin{gather}
\rT^{\theta}((X, \alpha), \delta)=  ((\rB(X),\mu^{\bB}_X), (\mu^{\bB}_{\uu}\tens \id_{\rB(X)})\rB_2(\rB(\uu),X)\rB(\delta)), 
\\  \rT^{\theta} (f) = \rT(f), \quad \Delta^{\bT^{\theta}}_{((X, \alpha), \delta)} = \Delta^{\bT}_{(X, \alpha)}, \quad  \varepsilon^{\bT^{\theta}}_{((X, \alpha),\delta)}= \varepsilon^{\bT}_{(X, \alpha)}.  \end{gather}

As in Section~\ref{descent-cats}, we use the identity natural isomorphism~$\tilde \Lambda \co \rF \rC \Rightarrow \rC^{\theta} \rF$ to exhibit an extension $\bT^{\theta}$ of $\bT$. The latter is explicitly given by  the formula $\tilde \Lambda_{((X,\alpha),\delta)} = \id_{(\rB(X), \mu^{\bB}_X)}$. 
According to Lemma~\ref{mate-computation}, the natural transformation $\tilde \Omega \co \rT^{\theta} \rG \Rightarrow \rG\rT$ whose mate equals to $\tilde \Lambda$ is computed by the formula $\tilde \Omega_{(X,\alpha)}= (\mu^{\bB}_{\uu}\tens \id_{\rB(X)})\rB_2(\rB(\uu),X)$.

\begin{cor} \label{desc-cat-cor-opmonoidal}
If~$\bB$ is a right pre-Hopf monad, then the natural transformation~$\tilde \Omega$ implements an extension~$\bT^{\theta}$ of~$\bT$. 
In that case, the left $\chi$-coalgebra structures on 
$\rN $ correspond bijectively to~$\mathbb{T}^{\theta}$-opcoalgebra structures on $\rN\rV$.
\end{cor}
 
\begin{proof} This follows by applying Corollary~\ref{desc-cat-cor} and by noting that $\tilde \Omega_{(X,\alpha)}$ equals to $H^r_{\uu,X}$ for any~$(X,\alpha) \in \Ob(\mc A^{\bB})$ .
\end{proof}

Following \cite[Section 5.2.]{MW10}, consider the functor~$\rK \co \mc A \to (\mc A^{\bB})_{\bC^{\theta}}$, given by setting for all~$X \in \Ob(\mc A)$ and morphisms~$f$  in $\mc A$, \[\rK(X)= (\rF(X), \rB_2(\uu, X)) \quad \text{and} \quad \rK(f)= \rF(f).\] 
Note that it follows from equation~\eqref{opmonoidality-mult} that for any~$X \in \Ob( \mc A)$, the morphism~$\rB_2(\uu, X)$ is a morphism from  the algebra~$\rF(X)$ to  the algebra~$\rC^{\theta}\rF(X)$. 
Thus the functor~$\rK$ is well-defined and induces a coalgebra~$(\rF, \gamma)$ over~$\bC^{\theta}$ by setting~$\gamma_X= \rB_2(\uu, X)$ for all $X \in \Ob(\mc A)$. 
We clearly have $\rV \rK = \rF$ and the theory from Section \ref{generalised-hopf-thm} applies. 
The Galois map associated to $\rK$ and~$\rF \dashv \rU$ is given by setting for all $(X,\alpha)\in \Ob(\mc A^{\bB})$, \[ \kappa_{(X,\alpha)}= (\id_{\rB(\uu)} \tens \alpha)\rB_2(\uu, X) \co (\rB(X), \mu^{\bB}_X) \to\rC^{\theta} (X,\alpha).\]

\begin{cor} If the opmonoidal monad~$\bB$ is a pre-Hopf monad and if~$\rF$ is comonadic, then left $\chi$-coalgebras on $\rN$ correspond to~$\overline {\bT^{\theta}}$-opcoalgebras on~$\rN\rF$. 
\end{cor}
\begin{proof}  It is observed in~\cite[Remark 3.13]{MW14} that the natural transformation~$\kappa\rF$ equals the left fusion operator~$ H^l_{\uu,-}$.  Since $\bB$ is a left pre-Hopf monad, $\kappa \rF$ is an isomorphism. Further, since $\rF$ is comonadic, it follows from Theorem \ref{generalised-hopf-thm-mw} that $\rK$ is an equivalence. The statement now follows by combining Corollaries \ref{desc-cat-cor-opmonoidal} and \ref{conjugate}.
\end{proof}

\section{An example: Hopf
algebras}\label{exfrombialgs}
In this section, we study a
particular and motivating
example which is
a special case of both
Sections \ref{bimonads-mw} and
\ref{bimonads-blv}. Let $H =
(H, \mu,  \eta, \Delta,
\varepsilon)$ be a
bialgebra over a commutative
ring~$k$, $\mc A = {}_{k}\Mod$
be the cate\-gory of left
$k$-modules,~$\mc B=
{}_{H}\Mod$ be the cate\-gory
of left $H$-modules, and  $\mc
C= {}_{H}^{H}\Mod$ be the cate\-gory of (left-left) Hopf modules over $H$.
Throughout the section we
denote $\tens= \tens_{k}$ and
$1= \eta(1_k)$. We  also use
Sweedler's notation for the
comultiplication, that is, we
write~$\Delta(h)= h_1\tens
h_2$ for any~$h\in H$. Also,
for an~$H$-coaction $\gamma$ of
a left $H$-comodule $M$, we
write $\gamma(m) = m_{-1}
\tens m_0$ for any $m \in M$.
The image of~$h\tens m$ under
a left $H$-action on $M \in
\Ob({}_H \Mod)$ will be
shortly written as $hm$. For
background information, see
\cite{radford}.

Furthermore, let $\mc Y = \bullet$ be the terminal category and let $\rM \co \mc Y \to  {}_{H}\Mod$ be the functor which picks a left $H$-module $M$. 
Finally, let $\mc Z= {}_{k}\Mod$ and set $\rN = N \tens_H - \co {}_{H}\Mod \to {}_{k}\Mod $, where $N$ is a right $H$-module. 
As we shall explain, this
yields the following
particular case of
diagram~\eqref{setup-desc-cats}
if $H$ is a Hopf algebra with
bijective antipode:
\begin{equation}  
\begin{gathered}
\xymatrix@C=8mm@R=4mm{
 & & {}_{H}^{H}\Mod \ar@/^-3mm/[dd]_\rV  \ar@{}[rrdd]^(.25){}="a"^(.75){}="b" \ar@{=>}^-{\tilde \Omega} "a";"b" \ar[rr]^{\rT^{\theta}}   \ar@{}[dd]|{\dashv}
 & &  {}_{H}^{H}\Mod \ar@/^-3mm/[dd]_\rV \ar@{}[dd]|{\dashv} && 
\\\\
	\bullet \ar[rr]^{\rM \quad} & & {}_{H}\Mod  \ar@/^-3mm/[uu]_\rG   \ar@/^3mm/[dd]^\rU \ar@{}[dd]|{\dashv} \ar@<0.4ex>[rr]^{\rT} \ar@<-0.4ex>[rr]_{\rC^{\theta}}
 & &{}_{H}\Mod
\ar@/^-3mm/[uu]_\rG
\ar@/^3mm/[dd]^\rU
\ar@{}[dd]|{\dashv}
\ar[rr]^{\rN = N \tens_H {-}}
&& {}_{k}\Mod.
\\\\
& & {}_{k}\Mod \ar@/^3mm/[uu]^\rF \ar@<0.4ex>[rr]^{\rB} \ar@{}[rruu]^(.25){}="a"^(.75){}="b" \ar@{=>}_-{\phantom{a} \Omega} "a";"b"  \ar@<-0.4ex>[rr]_{\rC}                   & & {}_{k}\Mod 
\ar@/^3mm/[uu]^\rF &&}
\end{gathered}
\end{equation}

\subsection{A bimonad from
bialgebras} In what follows,
we consider the bimonad
$\bB=(\rB,\mu^{\bB},
\eta^{\bB}, \Delta^{\bB},
\varepsilon^{\bB}, \theta)$
on ${}_{k}\Mod$ (in the sense
of \cite{MW10}), which is
defined by the formulas \begin{gather*} \rB=H \tens {-} , \quad \mu_X^{\bB}= \mu \tens \id_X, \quad \eta_X^{\bB}= 1 \tens  \id_X, \quad \Delta_X^{\bB}= \Delta\tens \id_X, \\
 \varepsilon_X^{\bB} = \varepsilon \tens \id_X, \quad \theta_M (a\tens b \tens m) = a_1b \tens a_2 \tens m.
\end{gather*}

\subsection{The first
adjunction} First, we have
that
${}_k\Mod^{\bB}={}_{H}\Mod$
and that the Eilenberg--Moore
adjunction $\rF \dashv \rU$
for~$\bB$ consists of the free
module functor~$\rF=H\tens -
\co {}_{k} \Mod \to
{}_{H}\Mod$ and the forgetful
functor $\rU$.  
For a left $H$-module $M$, the~$H$-action on $\rF(M)=H\tens M$ is induced by multiplication
in~$H$. 
The unit and the counit of the adjunction $\rF\dashv \rU$ are defined by the formulas
\[ \eta_M' (m) = 1 \tens m, \qquad \varepsilon_{(M,\alpha)}' (h \tens m) = hm. \]
Thus  the induced comonad $\bT$ on ${}_H \Mod$ is given by the formulas
\[\rT=\rF\rU = H\tens -, \quad \Delta^{\bT}_{(M, \alpha)}(h\tens m) = h \tens 1 \tens m, \quad \varepsilon^{\bT}_{(M,\alpha)} (h\tens m)=hm.\]
The second comonad $\bC^{\theta}$ on ${}_H\Mod$ is a lift of the comonad $\bC = (\rB, \Delta^{\bB}, \varepsilon^{\bB})$ on~${}_k\Mod$. It is explicitly given by $\rC^{\theta}(M,\alpha)= (H\tens M, \beta) $, where
$$ \beta (a\tens b \tens m) = a_1b \tens \alpha(a_2\tens m) = a_1b \tens a_2m$$ is the diagonal action. The comultiplication and the counit of the comonad~$\bC^{\theta}$ are defined by 
\[\Delta_X^{\bC^{\theta}}= \Delta \tens \id_X, \quad
 \varepsilon_X^{\bC^{\theta}} = \varepsilon \tens \id_X. \]
 The natural isomorphism $\Omega \co \rC \rU \Rightarrow \rU \rC^{\theta}$ which implements this lift is given by $\Omega_{(M,\alpha)} = \id_{H\tens M}$ for any $H$-module $M$. 
 Its mate $\Lambda \co \rF \rC \Rightarrow \rC^\theta \rF$ is given by setting for any~$k$-module $M$,  \[\Lambda_M (a \tens b\tens m)= a_1b\tens a_2 \tens m,\]
 which is an $H$-module map from the free $H$-module~$H\tens H \tens M$ to the~$H$-module~$H\tens H \tens M$ with the diagonal action on first two tensorands.

\subsection{The second adjunction} We clearly have that $({}_H\Mod)_{\bC^{\theta}}= {}_{H}^{H}\Mod  $. The Eilenberg--Moore adjunction $\rV \dashv \rG$ for $\bC^{\theta}$ consists of the forgetful functor 
$\rV \co {}_{H}^{H}\Mod \to
{}_{H}\Mod$ (which is left
adjoint!) and the cofree comodule functor $\rG \co  {}_{H}\Mod \to {}_{H}^{H}\Mod $.
More precisely,  $\rG$ is defined by setting
$\rG= H \tens -$, such that for an $H$-module~$M$, the $H$-action on~$\rG(M)$ is the diagonal action 
and the $H$-coaction on~$\rG(M)$ is given by~$\Delta \tens \id_M$.
The unit and the counit of the
adjunction $\rV \dashv \rG$
are defined by the formulas 
\[\eta''_{(M, \alpha, \gamma)}  = \gamma, \qquad \varepsilon''_{(M,\alpha)} = \varepsilon^{\bC^{\theta}}_{(M,\alpha)}. \]
The comonad $\bT^{\theta}$
such that $\rV \rT^{\theta} =
\rT \rV$ is given by
\[\rT^{\theta}= H\tens -, \quad \Delta^{\bT^{\theta}}_{(M, \alpha, \gamma)}(h\tens m) = h \tens 1 \tens m, \quad \varepsilon^{\bT^{\theta}}_{(M,\alpha)} (h\tens m)=hm.\]
For a Hopf module $M$, the $H$-action on $\rT^{\theta}(M)$ is the free action and the~$H$-coaction~$\gamma$  on $\rT^{\theta}(M)$
is given by the formula \[\gamma (h\tens m) = h_1m_{-1} \tens h_2 \tens m_0.  \]
We have the natural
transformation $\tilde
\Lambda \co  \rV\rT^{\theta}
\Rightarrow \rT \rV$, which is
given by setting $\tilde
\Lambda _{(M,\alpha,\gamma)} =
\id_{\rF(M)}$ for any
$H$-module $M$. The following
result allows us to define
$\tilde \Omega \co \rT^\theta
\rG \Rightarrow \rG \rT$,
which will be part of that
data for the extension~$\bT^\theta$ of $\bT$ through $\rV \dashv \rG$.  

\begin{lem} \label{hopfbij}
The following statements hold:
\begin{enumerate}
  \item      The unique
natural transformation $\tilde
\Omega \co \rT^{\theta} \rG
\Rightarrow \rG\rT$ whose mate
equals~$\tilde \Lambda$ is
given for any $(M,\alpha)$ in
$_{H}\Mod$ by
      \[\tilde \Omega_{(M,\alpha)} (a \tens b \tens m) = \theta_M (a \tens b \tens m)= a_1b \tens a_2 \tens m. \]
  \item The morphism $\tilde \Omega_{(M,\alpha)}$ is a morphism of Hopf modules for any $(M,\alpha)$ in $_{H}\Mod$. 
  \item If $H$ is a Hopf algebra with a bijective antipode, then $\theta_M$ is invertible with inverse
\[\theta^{-1}_M (a\tens b \tens m)= b_2\tens S^{-1}(b_1)a \tens m.
\]  Moreover,~$\tilde
\Omega$ implements an
extension of $\bT^{\theta}$ of
$\bT$ through  $\rV \dashv \rG$.  
\end{enumerate}
\end{lem}

\begin{proof} This is mostly a corollary of statements seen in Section \ref{descent-cats} which apply to bimonads in the sense of Mesablishvili and Wisbauer.
\begin{enumerate}
  \item This statement follows from Lemma
\ref{mate-computation}.
  \item Following $(1)$, one
needs to verify that
$\theta_M$ is a morphism from
the free $H$-module $H\tens
H\tens M$ to the $H$-module
$H\tens H\tens M$, where~$H$
acts diagonally on the first
two tensorands.  Also, one
needs to check that it is a
morphism from~the
$H$-comodule~$H\tens H \tens
M$ with coaction $$a\tens b
\tens m \mapsto a_1b_1\tens
a_2 \tens b_2 \tens m$$  to
the cofree comodule $H\tens H \tens M$. Both verifications can be done by hand but they also follow from the fact that $\theta$ is a mixed distributive law.
  \item The first part of the statement  is well-known and left to the reader. See~\cite[Example 2.10]{BLV11} for the expression of the inverse
in the case of Hopf algebras (with invertible antipode) in braided monoidal categories.
Finally, once $\theta$ is an
isomorphism, it follows from
$(1)$ and~$(2)$ that~$\tilde\Omega_{(M,\alpha)}$ is
an isomorphism of the
aforementioned Hopf modules for any $H$-module $(M,\alpha)$. \qedhere
\end{enumerate}
\end{proof}

Let
$\rK \co {}_k \Mod \to {}_{H}^{H}\Mod $ be the functor given by the formulas
 \[\rK(X)= (\rF(X), \Delta_X^{\bB})= ((\rB(X),\mu^{\bB}_X), \Delta_X^{\bB}) \quad \text{and} \quad \rK(f)= \rF(f).\] 
It follows from Section \ref{bimonads-mw}
 that the functor~$\rK$ is well-defined 
 and induces a coalgebra~$(\rF, \gamma)$ 
 over~$\bC^{\theta}$ by setting~$\gamma_X= \Delta_X^{\bB}$ for all $k$-modules $X$.
The Galois map associated to~$\rK$ and~$\rF \dashv \rU$
 is given by setting for all $H$-modules~
 $(M,\alpha)$, \[ \kappa_{(M,\alpha)}= \rB(\alpha)\Delta_M^{\bB} \co (\rB(M), \mu^{\bB}_M) \to (\rB(M), \rB(\alpha)\theta_M).\] 
For any $k$-module $X$, the map $\kappa_{\rF(X)} \co H\tens H \tens X \to H\tens H \tens X$ is given by setting for any $h,\ell \in H$ and $x \in X$, $$\kappa_{(H\tens X, \mu_H \tens \id_X)} (h\tens \ell \tens x) = h_1\tens h_2 \ell \tens x.$$
It is well-known that $H$ is a
Hopf algebra if and only if the
latter map is for all $X$ an
isomorphism; in this case, its
inverse is 
$h \otimes \ell  \tens x \mapsto 
h_1 \otimes S(h_2)\ell \tens x$. 
By the fundamental
theorem on Hopf modules from
\cite{LS69}, this is also
equivalent to the condition
that the category $\mc C =
{}_{H}^{H}\Mod$ of Hopf modules is equivalent
to~$\mc A = {}_k \Mod$ by
means of the functor that
assigns to a Hopf module $M$
the space of coinvariants $k \Box_H M = \{ x
\in M \mid x_{-1} \otimes x_0
= 1 \otimes x\}$.

If $H$ is a bialgebra, right-$\chi
$-coalgebras are by \cite[Proposition 3.6]{KKS15} in bijection with $H$-modules
which are simultaneously $H$-comodules (without a compatibility condition between them).
However, if $H$ is
a Hopf algebra with bijective
antipode, thanks to Lemma \ref{hopfbij} (3) we can apply Corol\-laries \ref{desc-cat-cor} and \ref{conjugate} to obtain a
completely dual classification
of the left $ \chi
$-coalgebras as modules and
comodules.

\begin{rem}
The example of this section
can be explained in terms of
the opmonoidal monad (see Section \ref{bimonads-blv}) $\rB=
H\tens -$ on ${}_k\Mod$, where
$$\rB_2(X,Y) \co H\tens X
\tens Y \to H \tens X \tens
H\tens Y,\quad\text{and}\quad \rB_0 \co H \to
k$$ 
are given by $$\rB_2(X,Y)(h \tens x \tens y) = h_1 \tens x \tens h_2 \tens y \quad \text{and} \quad \rB_0 (x) = \varepsilon(x).$$ 
\end{rem}

\section{A nonexample: The list bimonad}

In this section, we explore the non-empty list monad 
on the category $\cSet$ of
sets. This is a bimonad in the
sense of~\cite{MW11}. However, we show that in this case the important Corollaries \ref{desc-cat-cor} and \ref{conjugate} from Section \ref{mainresult} are not applicable.

\subsection{The non-empty list bimonad}

In order to fix the notation, we recall the following.
\begin{defn}A \emph{list} on a set
$X$ is an element of the set
$\coprod_{n \ge 0} {X^n}$, that is,
a pair $(l,n)$, where $l\in X^n$ is an ordered $n$-tuple of elements 
in $X$ and $n \ge 0$ is an
integer called the \emph{length} of
the list $l$. The list is
\emph{non-empty} if $n  > 0$.
In what follows, we denote such a list $((x_1, \dots, x_n),n)$ by $[x_1, \dots, x_n]$.
\end{defn} 

Following
\cite[Section 6.3]{KKS15}, one
defines the \emph{non-empty
list bimonad} $\bL^+ = (\rL^+,
\mu^{\bL^+}, \eta^{\bL^+},
\Delta^{\bL^+},
\varepsilon^{\bL^+}, \theta) $ on~$\cSet$  as follows. First, for any set~$X$ and 
any function~$f\co X \to Y$, 
$$\rL^+ (X) = \coprod_{n>0} {X^n} \quad \text{and} \quad 
\rL^+ (f) ([x_1, \dots, x_n]) = [f(x_1), \dots, f(x_n)].$$
Next, the multiplication
$\mu^{\bL^+}_X \co \rL^+(\rL^+ (X)) \to \rL^+(X) $ on $\bL^+$
is given by sending a list
$$ [ [ x_{1,1}, \ldots, x_{1,
n_1} ] , \ldots, [x_{m,1}, \ldots, x_{m, n_m}]] $$ to the concatenated list
$$ [x_{1,1}, \ldots, x_{1,
n_1} ,\ldots, x_{m,1}, \ldots, x_{m, n_m}]. $$
The unit map
$\eta^{\bL^+}_X
 \co X \to \rL^+ (X)$ sends an
element~$x\in X$ to the list
$[x]$. The comultiplication
and counit of $\bL^+$ are
given by the formulas 
$$
	\Delta^{\bL^+}_X( [x_1, \ldots, x_n])
= [[x_1, \ldots, x_n], \ldots, [x_n]], \quad 
	\varepsilon^{\bL^+}_X([x_1, \ldots, x_n]) = x_1.
$$  
To shorten the notation, we sometimes refer to the comonad $(\rL^+, \Delta^{\bL^+}, \varepsilon^{\bL^+})$ as $\bC$.
Finally, the mixed
distributive law
$ \theta$  is defined as follows.
 For any set~$X$, the function $\theta_X$ sends a list
$$ [ [ x_{1,1}, \ldots, x_{1,
n_1} ] , \ldots, [x_{m,1}, \ldots, x_{m, n_m}]] $$ in
$\rL^+(\rL^+ (X))$ to the list \begin{align*}
\Big[[x_{1,1}, x_{2,1}, x_{3,1} \ldots, x_{m,1}],
\ldots, [x_{1, n_1}, x_{2,1}, x_{3,1}, \ldots,
x_{m,1}]&, \\ [x_{2,1}, x_{3,1} \ldots, x_{m,1}],
\ldots, [x_{2, n_2}, x_{3,1}, \ldots x_{m,1} ]&,\\
\ldots&,  \\  [ x_{m,1} ] , [ x_{m,2} ], \ldots,
[x_{m,n_m} ] \Big]&
\end{align*}
in $\rL^+(\rL^+ (X))$.

\begin{prop} \label{theta-not-inv}
For any set $X$, the function $\theta_X$ is injective. However, for a non-empty set $X$, the function $\theta_X$ is not surjective.
\end{prop}
\begin{proof} The injectivity follows from the definition of the mixed distributive law~$\theta$. 
To see that~$\theta_X$ is not a surjection we remark that for any non-empty set~$X$, there is a list $y \in \rL^+ (\rL^+ (X))$  which has no preimage. 
For instance, take~$y= [[y_1, y_2]] \in \rL^+(\rL^+(X)).$ 
Assume that there is an~$x \in \rL^+ (\rL^+ (X))$ with $\theta_X(x) =y$. 
The list~$x$ is of the form 
$$ [ [ x_{1,1}, \ldots, x_{1,n_1} ] , \ldots, [x_{m,1}, \ldots, x_{m, n_m}]] $$ for some $m\ge 1$ and $n_1, \dots, n_m \ge 1$.
Since~$\theta_X(x) =y$, we necessarily have $\mathrm{length}(\theta_X(x))= \mathrm{length}(y),$ which is equivalent with 
$$n_1+ \dots + n_m =1.$$ Now since the left hand side of the above equation is at least~$m\ge 1$, we have that~$m = 1$ and $n_1=1$. Contradiction.
\end{proof}

\subsection{Entwined coalgebras}
The Eilen\-berg--Moore cate\-gory~$\cSet^{\bL^+}$ of the monad $\bL^+= (\rL^+, \mu^{\bL^+}, \eta^{\bL^+})$ is isomorphic to the cate\-gory $\cSgp$ of semigroups (see \cite[Proposition 6.4.]{KKS15}). 
The Eilenberg--Moore adjunction~$\rF \dashv \rU$ for the monad $\bL^+$ consists of the free functor, which is given by $$  \rF(X) = (\rL^+(X), \mu_X^{\bL^+}), \qquad \rF(f) = \rL^+ (f)$$ and the forgetful functor $\rU \co \cSgp \to \cSet$.
The unit and the counit of the adjunction $\rF \dashv \rU$ are given by the formulas
$$\eta^{\bL^+}_{X}(x)=[x], \qquad \varepsilon^{\bL^+}_{(X, \alpha)} = \alpha.$$
Clearly, $\rU\rF = \rL^+$, but by Section \ref{monads-and-adjunctions}, we also have that the endo\-functor $\rF \rU$ is part of a comonad $\bT$ on $\cSgp$.
Next, following Section \ref{descent-cats},  there is a comonad lift $\bC^{\theta}$ of $\bC= (\bL^+, \Delta^{\bL^+}, \varepsilon^{\bL^+})$ through the adjunction $\rF \dashv \rU$ so that 
\[
\xymatrix{\cSgp \ar[r]^{\rC^{\theta}}  \ar[d]_{\rU} &   \cSgp \ar[d]^{\rU}\\
	\cSet \ar@{}[ru]^(.25){}="a"^(.75){}="b" \ar@{=>}^-{\Omega} "a";"b"  \ar[r]_{\rC}  & \cSet}
\] commutes up to the natural isomorphism $\Omega=\id$, i.e. commutes strictly. Further, let $\rV \dashv \rG$ be the Eilenberg--Moore adjunction for $\bC^{\theta}$. As described in Section \ref{descent-cats}, there is a comonad $\bT^\theta$ on the category ~$\cSgp_{\bC^{\theta}}$ such that $\rV \rT^\theta = \rT \rV$. 
This defines the identity natural transformation $\tilde \Lambda \co  \rV \rT^\theta \Rightarrow \rT \rV$. 
 We have by Lemma~\ref{mate-computation}, that the unique natural transformation $\tilde \Omega \co \rT^{\theta} \rG \Rightarrow \rG\rT$ whose mate equals~$\tilde \Lambda$ is given by $\tilde \Omega_{(X,\alpha)}=\theta_X$. 
 Since $\theta$ in not an isomorphism (see Proposition~\ref{theta-not-inv}), $\tilde{\Omega}$ cannot implement an extension $\bT^{\theta}$ of $\bT$. Thus the Corollary \ref{desc-cat-cor} is not applicable.

\subsection{Non-invertibility of the Galois map}

Let
$\rK \co \cSet \to \cSgp_{\bC^{\theta}}$ be the functor given by the formulas
 \[\rK(X)= (\rF(X), \Delta_X^{\bL^+})= ((\rL^+(X),\mu^{\bL^+}_X), \Delta_X^{\bL^+}) \quad \text{and} \quad \rK(f)= \rF(f).\] 
It follows from Section \ref{bimonads-mw} that the functor~$\rK$ is well-defined and induces a coalgebra~$(\rF, \gamma)$ over~$\bC^{\theta}$ by setting~$\gamma_X= \Delta_X^{\bL^+}$ for all sets $X$.
The Galois map associated to~$\rK$ and~$\rF \dashv \rU$ is given by setting for all semigroups~$(X,\alpha)$, \[ \kappa_{(X,\alpha)}= \rL^+(\alpha)\Delta_X^{\bL^+} \co (\rL^+(X), \mu^{\bL^+}_X) \to (\rL^+(X), \rL^+(\alpha)\theta_X).\] 
Explicitly, $\kappa_{(X,\alpha)} ([x_1,\ldots, x_n]) = [x_1\cdots x_n, x_2\cdots x_n, \ldots, x_n]$.
In the following proposition we characterize semigroups $(X,\alpha)$ with the property that $\kappa_{(X, \alpha)}$ is bijective.
\begin{prop} \label{galois-not-inv} For a nonempty semigroup $(X,\alpha)$, the semigroup morphism $\kappa_{(X,\alpha)}$ is bijective if and only if $(X,\alpha)$ is a group. 
\end{prop}
\begin{proof}
It is straightforward to show that if~$(X,\alpha)$ is a
group, then  
$ \kappa _{(X, \alpha )}$ is
bijective.
Conversely, let $(X,\alpha)$ be a semigroup such that $\kappa _{(X, \alpha )}$ is bijective.
An elementary argument shows that the bijectivity of   $\kappa _{(X, \alpha )}$
is equivalent to the bijectivity of the map  $ [x,y] \mapsto 
[xy,y]$.
The latter in turn
is injective if and only if
the semigroup is right cancellative, 
that is, if for all~$x,y,z \in X$,
we have $xz = yz \Rightarrow x=y$. 

First, since
$[x,y] \mapsto [xy,y]$ is
bijective, we have that  for every  
$ y \in X$ there is 
a unique element $e_y \in X$
with $
\kappa _{(X, \alpha )} ( [e_y, y]
) = [y,y]$, that is, such that~$e_y y = y$. 
Note that~$(e_y e_y) y = e_y (e_yy) = 
e_y y = y$, so the uniqueness
implies~$e_y e_y = e_y$. 
Furthermore, the equality  
$(ye_y) y = y (e_yy) = yy$ combined with the fact that $X$ is right cancellative  
shows $ye_y = y$. In other words, each $y
\in Y$ has a unique idempotent 
local unit
element $e_y \in X$ with 
$e_y y = y e_y = y$. 
Next, for all $y \in
X$ there is a unique 
$\bar y\in X$ 
with
$ \kappa _{(X, \alpha )} ([\bar y,
y]) = [e_y,y]$, that is, such that 
$\bar y y = e_y$.
It follows that $(y \bar y) y = y (\bar y y) = 
y e_y = y = e_y y$, so~$y \bar y =e_y$. 
Thus $y \in X$ 
has a unique local inverse 
$\bar y \in Y$ such that~$\bar y y = y \bar y = e_y$. 

Now observe that 
$(e_y \bar y) y=
e_y (\bar y y) = 
e_y e_y = e_y = \bar y y$, 
so 
$e_y \bar y = \bar y$, which means~$e_{\bar y} = e_y$. So $y \bar y = 
\bar y y = e_y$ also implies
$ \bar {\bar y} = y$. 
Moreover, the equality~$ e_y e_y = e_y$ implies 
$e_{e_y} = e_y$ and 
$ \overline{e_y} = e_y$.
If now $x,y \in X$ are two elements,
then~$(y e_x) x = y (e_xx)=  yx $, so cancelling
$x$ on the right yields~$ye_x =y$. 
That is, all local unit elements~$e_x$ are right unit elements
for all $y \in X$. However, this in
particular 
means that~$e_ye_x = e_x= e_xe_x$ which yields 
$ e_y = e_x$. Thus as long as $X$ is
not empty, the local units all agree
and it follows that $X$ is a group. 
\end{proof}

In particular, the Galois map $\kappa$ is not an isomorphism in general. As a corollary, we have: 

\begin{cor} The functor $\rK$ is not an equivalence.
\end{cor}

\begin{proof}
By Theorem \ref{generalised-hopf-thm-mw}, ~$\rK$ is an equivalence if and only if the functor~$\rF$ is comonadic and if~$\kappa\rF$ is an isomorphism. 
By \cite[Proposition 2.5]{MW14}, the invertibility of $\kappa\rF$ is equivalent to the invertibility of $\kappa$. 
However, by Proposition \ref{galois-not-inv}, $\kappa$ is not an isomorphism in general. Thus $\rK$ cannot be an equivalence.
\end{proof}

To sum up, Corollary \ref{conjugate} is not applicable in the case of the non-empty list bimonad $\bL^+=(\rL^+,
\mu^{\bL^+}, \eta^{\bL^+},
\Delta^{\bL^+},
\varepsilon^{\bL^+}, \theta)$. 
Also, since $\kappa \rF
%= \{\rL^+(\mu^{\bL^+}_X)\Delta_X^{\bL^+}\}_{X\in \Ob(\cSet) } 
$ is not an isomorphism, we have that it is not a Hopf monad in the sense of~\cite{MW11} (see Remark~\ref{hopf-monad-mw}).

\end{document}